\documentclass[12pt,reqno]{amsart}

\usepackage{epsfig}
\usepackage{amsmath, amsthm, amsfonts, amssymb, mathrsfs}
\usepackage{tikz-cd}
\usepackage{graphicx}

\numberwithin{equation}{section}

\newcommand{\R}{{\mathbb R}}
\newcommand{\C}{{\mathbb C}}
\newcommand{\N}{{\mathbb N}}

\newcommand{\Z}{{\mathbb Z}}

\newcommand{\fcal}{{\mathcal F}}
\newcommand{\hcal}{{\mathcal H}}
\newcommand{\kcal}{{\mathcal K}}
\newcommand{\ocal}{{\mathcal O}}
\newcommand{\pcal}{{\mathcal P}}

\newtheorem{theo}{{\sc \bf Theorem}}[section]

\newtheorem{lem}[theo]{{\sc \bf Lemma}}
\newtheorem{prop}[theo]{{\sc \bf Proposition}}

\begin{document}

\title{A Pile of Shifts I: Crossed Products}

\author[Hebert]{Shelley Hebert}
\address{Department of Mathematics and Statistics,
Mississippi State University,
175 President's Cir. Mississippi State, MS 39762, U.S.A.}
\email{sdh7@msstate.edu}

\author[Klimek]{Slawomir Klimek}
\address{Department of Mathematical Sciences,
Indiana University Indianapolis,
402 N. Blackford St., Indianapolis, IN 46202, U.S.A.}
\email{sklimek@math.iupui.edu}

\author[McBride]{Matt McBride}
\address{Department of Mathematics and Statistics,
Mississippi State University,
175 President's Cir., Mississippi State, MS 39762, U.S.A.}
\email{mmcbride@math.msstate.edu}

\date{\today}

\begin{abstract}
We discuss C$^*$-algebras associated with several different natural shifts on the Hilbert space of the $s$-adic tree, that is, the tree of balls in the space of $s$-adic integers. 

\end{abstract}

\maketitle
\section{Introduction}
It is well known that Cantor sets could be treated as the boundary of a directed rooted tree \cite{PB} via Michon correspondence.
Consider a $s$-regular infinite directed rooted tree called the $s$-adic tree. We adopt here the following definition of such a tree: all vertices have the out-degree equal to $s$ and all vertices except the root have the in-degree equal to one (the in-degree of the root is equal to zero).
For such a tree, the boundary is the set of infinite out-paths starting at the root. Those can be naturally identified with the ring of $s$-adic integers $\Z_s$, see equation \eqref{defn-sadic}. 

The vertices $\mathcal{V}$ of the $s$-adic tree can be thought of as corresponding to balls in $\Z_s$, with ball inclusion defining edges of the $s$-adic tree, which can be considered as a coarse grained approximation of the space of $s$-adic integers. The corresponding  Hilbert space $\ell^2(\mathcal{V})$ carries a straightforward, diagonal representation of the algebra of continuous functions on $\Z_s$.
In previous papers  \cite{KMR}, \cite{KRS1}, \cite{KRS2} of some of the authors, a version of a unilateral shift on $\ell^2(\mathcal{V})$ was used to construct a ``forward'' derivative leading to a spectral triple over $C(\Z_s)$.
However, there are other interesting shifts on that Hilbert space corresponding to different arithmetic and geometric aspects of the space of $s$-adic integers.

The purpose of this paper and the companion paper \cite{HKMP3} is to study natural C$^*$-algebras in $\ell^2(\mathcal{V})$ generated by $C(\Z_s)$ along with specific shifts. Those algebras are interesting examples of C$^*$-algebras with connections to number theory and combinatorics, a subject of much interest \cite{CL}, \cite{R}. 

Given a C$^*$-algebra $A$ and an endomorphism there is a number of ways to construct a new C$^*$-algebra, desirably an extension of $A$, called the crossed product. Often an intermediate extension procedure is needed to enlarge the algebra $A$ up to a better behaved ``coefficient"  C$^*$-algebra. This paper serves as a source of interesting examples of such constructions and we determine how each of these shift algebras can be identified with a crossed product by an endomorphism, in fact a monomorphism with a hereditary range. In the companion paper we delve deeper into the structure of the shift algebras and compute their K-Theory.

This paper is organized as follows. In the next section we introduce four shifts on the $s$-adic tree and the corresponding C$^*$-algebras. 
Then we discuss coefficient algebras and the related crossed products. Particularly useful here is the reference \cite{LO} which considers, in general, C$^*$-algebras generated by a given concrete C$^*$-algebra and a partial isometry. What makes our considerations special is the existence of what we call a ``gauge action'', a symmetry that leaves $C(\Z_s)$ invariant.

After the preliminaries, each algebra is studied in turn in its own dedicated section. While the algebras are quite different, there are number of similarities. 
Propositions \ref{BUProp}, \ref{BVProp}, \ref{BSProp} and \ref{BWProp} describe, in each case, generators of the minimal coefficient algebra extension of $C(\Z_s)$.
Also, analogous Propositions \ref{BUThm}, \ref{BVThm}, \ref{BSThm} and \ref{BWThm},  each describes the corresponding gauge invariant subalgebra through a ``Toeplitz'' map.
The main results contained in Theorems \ref{AUTheo},  \ref{AVTheo}, \ref{ASTheo} and \ref{AWTheo} describe each algebra as a crossed product by a monomorphism with hereditary range. They are a straightforward consequence of the results directly preceding the theorems and general considerations in Section 3.

\section{Trees and Shifts}
\subsection{$s$-adic Integers}
For $s\in\N$ with $s\ge2$, the space of $s$-adic integers $\Z_s$ is defined as a set consisting of infinite sums:
\begin{equation}\label{defn-sadic}
\Z_s= \left\{x=\sum_{j=0}^\infty x_j s^j: 0\le x_j\le s-1, j=0,1,2,\ldots\right\}\,.
\end{equation}

We call the above expansion of $x$ its $s$-adic expansion.  Note that $\Z_s$ is isomorphic with the countable product $\prod \Z/s\Z$ and so, when the product is equipped with the Tychonoff topology, it is a Cantor set. It is also a metric space with the usual norm $|\cdot|_s$ which is defined as follows: 
\begin{equation*}
|0|_s=0 \textrm{ and, if }x\ne 0, |x|_s=s^{-n},
\end{equation*}
where $x_n$ is the first nonzero term in the above $s$-adic expansion. The only possible distances are powers of $s$ and $0$ giving $\Z_s$  unusual properties specific to non-Archimedean metrics, see \cite{Ro}. In particular, every point in a ball is its center, two balls are either disjoint or one is contained in the other. Additionally, any ball of radius $s^{-n}$ is a union of precisely $s$ disjoint balls of radius $s^{-n-1}$. We also have that locally constant functions are dense in $C(\Z_s)$.

Moreover, $\Z_s$ is an Abelian ring with unity with respect to addition and multiplication with a carry and we can identify $\Z_s$ with the projective limit:
\begin{equation*}
\Z_s\cong\lim_{\longleftarrow}\Z/s^n\Z\,,
\end{equation*}
where the projections in the above system are given by reductions modulo powers of $s$.
Using the division algorithm modulo $s$ repeatedly, we can associate an $s$-adic expansion to any integer and so $\Z_s$ naturally contains a copy of $\Z$. For more details, see \cite{KM6}.

\subsection{Hilbert Spaces and Shifts}
We can naturally associate an infinite directed rooted tree to the ring $\Z_s$ called the $s$-adic tree. The set of vertices $\mathcal{V}$ of this tree is defined to be the set of all balls in $\Z_s$.  The set of edges of the $s$-adic tree is defined as follows: the edges are all ordered pairs of balls $(v,v')$ such that $v'\subsetneqq v$, and there is no nontrivial ball between $v$ and $v'$.  The root is the ball of radius one, which is $\Z_s$, and the structure of the balls implies that all vertices have out-degree equal to $s$ and all vertices except the root have in-degree equal to one, while the in-degree of the root is equal to zero.

In fact, by choosing a particular integer center of a ball, we see that any ball is of the form:
\begin{equation*}
\left\{y\in\Z_s: |y-x|_s\le\frac{1}{s^n}\right\}
\end{equation*}
where $n\in\Z_{\ge0}$ and $0\le x<s^n$.  Consequently we can write $\mathcal{V}$ in the following way:
\begin{equation*}
\mathcal{V}=\{(n,x): n=0,1,2,\ldots,0\le x<s^n\}\,.
\end{equation*}

Throughout this paper we use the Hilbert space $H=\ell^2(\mathcal{V})$ with canonical basis $\{E_{(n,x)}\}$.  Let $f\in C(\Z_s)$ and define a bounded operator $M_f:H\to H$ by
\begin{equation*}
M_fE_{(n,x)}=f(x)E_{(n,x)}\,.
\end{equation*}
Notice, due to the density of $\Z_{\ge0}$ in $\Z_s$, that we have that
\begin{equation*}
\|M_f\| = \sup_{x\in\Z_{\ge0}}{|f(x)|} = \sup_{x\in\Z_s}{|f(x)|}=||f||_\infty,
\end{equation*}
and so $f\mapsto M_f$ is a faithful representation of the C$^*$-algebra $C(\Z_s)$.

There are four main shifts acting on $H$ that are the focus of this paper and which we call the Bunce-Deddens shift, the Hensel shift, the Bernoulli shift, and the Serre shift.  We use the word ``shift" as those operators resemble the usual unilateral shift. All four shifts are defined on $H$ with formulas given on the basis elements as follows:

\begin{enumerate}
\item \underline{The Bunce-Deddens Shift:} $U:H\to H$ via 
\begin{equation*}
UE_{(n,x)}= E_{(n+1,\,x+1)}\,.
\end{equation*}  
\item \underline{The Hensel Shift:} $V:H\to H$ via 
\begin{equation*}
VE_{(n,x)} = E_{(n+1,\,sx)}\,.
\end{equation*}
\item \underline{The Bernoulli Shift:} $S:H\to H$ via 
\begin{equation*}
SE_{(n,x)}=\frac{1}{s}\sum_{j=0}^{s-1}E_{(n+1,\,sx+j)}\,.
\end{equation*}
\item \underline{The Serre Shift:} $W:H\to H$ via
\begin{equation*}
WE_{(n,x)} = \frac{1}{\sqrt{s}}\sum_{j=0}^{s-1}E_{(n+1,\,x+js^n)}\,.
\end{equation*}
\end{enumerate}
Let us comment on those definitions. The Bunce-Deddens shift is well-defined, since if $x<s^n$ then $x+1<s^{n+1}$. This particular shift uses only the additive structure of $\Z_s$.  For the Hensel shift, if $x<s^n$, then $sx<s^{n+1}$ and this shift uses only the multiplicative structure of $\Z_s$.  The Bernoulli shift is another arithmetic shift, using both the additive and multiplicative structures of $\Z_s$: if $x\leq s^n-1$ and $j\leq s-1$ then $sx+j\leq s^{n+1}-1$.
The Serre shift, used in \cite{KMR}, \cite{KRS1}, \cite{KRS2}, is essentially a shift along the edges of the $s$-adic tree. It is well defined since if $x\leq s^n-1$ and $j\leq s-1$ then $x+js^n\leq s^{n+1}-1$.

A computation yields their respective adjoints, where care has to be taken to find the kernels of those adjoints:
\begin{enumerate}
\item \begin{equation*}
U^*E_{(n,x)} = \left\{
\begin{aligned}
&E_{(n-1,\,x-1)} &&\textrm{ if }n\ge1, 0<x\le s^{n-1} \\
&0 &&\textrm{ if }n=0\textrm{ or }x=0\textrm{ or }s^{n-1}<x<s^n
\end{aligned}\right.
\end{equation*}
\item \begin{equation*}
V^*E_{(n,x)} = \left\{
\begin{aligned}
&E_{(n-1,\,x/s)} &&\textrm{ if }n\ge1,\, s\mid x\\
&0 &&\textrm{ if }n=0\textrm{ or }s\nmid x
\end{aligned}\right.
\end{equation*}
\item \begin{equation*}
S^*E_{(n,x)} = \left\{
\begin{aligned}
&E_{(n-1,\,(x-x\textrm{ mod }s)/s)} &&\textrm{ if }n\ge1\\
&0&&\textrm{ if } n=0
\end{aligned}\right.
\end{equation*}
\item \begin{equation*}
W^*E_{(n,x)} = \left\{
\begin{aligned}
&\frac{1}{\sqrt{s}}E_{(n-1,\,x\textrm{ mod }s^{n-1})} &&\textrm{ if }n\ge1\\
&0&&\textrm{ if }n=0.
\end{aligned}\right.
\end{equation*}
\end{enumerate}
It is easy to verify that 
$$U^*U=V^*V=S^*S=W^*W=I,$$ 
hence each one is an isometry.  For each of these four shifts $\mathcal{J}\in\{U, V, S, W\}$ we have the corresponding C$^*$-algebras, which are the main object of study of this paper:
\begin{equation*}
A_{\mathcal{J}} := C^*(\mathcal{J},M_f:  f\in C(\Z_s))\,.
\end{equation*}

There is an important structure shared by all these algebras which, by analogy with graph algebras \cite{R}, we call a gauge action. For $\theta\in\R/\Z$, consider the following strongly continuous one-parameter family $\{\mathcal{U}_\theta\}$ of unitary operators on $H$ defined by
\begin{equation}
\mathcal{U}_\theta E_{(n,x)}=e^{2\pi in\theta}E_{(n,x)}.
\end{equation}\label{Udefref}
We define a one-parameter family $\{\rho_\theta\}$ of automorphisms, the \textit{gauge action}  $\rho_\theta:A_{\mathcal{J}}\to A_{\mathcal{J}}$  by the formula
\begin{equation*}
\rho_\theta(a)=\mathcal{U}_\theta a \mathcal{U}_\theta^{-1}.
\end{equation*}
Then on generators of $A_{\mathcal{J}}$ we have 
\begin{equation*}
\rho_\theta(\mathcal{J})=e^{2\pi i\theta}\mathcal{J}\,,\quad \rho_\theta(\mathcal{J}^*)=e^{-2\pi i\theta}\mathcal{J}^*\,,\quad\textrm{and}\quad \rho_\theta(M_f) = M_f\,.
\end{equation*}
It easily follows by an approximation argument that for every $a\in A_{\mathcal{J}}$ the map $\theta\mapsto \rho_\theta(a)$ is continuous.

\section{Coefficient algebras with a gauge action}
\subsection{Coefficient algebras}
The notion of  a coefficient algebra was introduced in  \cite{LO}.  A review of this concept, with additional material suitable for our purposes, is the main theme of this section.

Given a Hilbert space $\hcal$, let $A \subseteq B(\hcal)$ be a C$^*$-subalgebra of $B(\hcal)$ containing the identity $I$ and let ${\mathcal{J}} \in B(\hcal)$ be an isometry, ${\mathcal{J}}^*{\mathcal{J}}=I$. We use the notation  $A_{\mathcal{J}}$ for the C$^*$-algebra generated by $A$ and ${\mathcal{J}}$. We call $A$ a {\it coefficient algebra} of the C$^*$-algebra $A_{\mathcal{J}}$  if $A$ and ${\mathcal{J}}$ satisfy the following  conditions:
\begin{equation*}
\alpha(a):={\mathcal{J}} a{\mathcal{J}}^* \in A,
\end{equation*}
and
\begin{equation*}
\beta(a):={\mathcal{J}}^*a{\mathcal{J}} \in A,
\end{equation*}
for all $a\in A$.

It is easy to see that $\alpha$ is an endomorphism of $A$ and we have the following properties:
\begin{equation*}
\beta\alpha(a)=a,\, \alpha\beta(a)=\alpha(I)a\alpha(I), \textrm{ and } \beta(I)=I.
\end{equation*}
Moreover,  the map $\beta: A\to A$ is linear, continuous, positive and has the ``transfer'' property:
\begin{equation*}
\beta(\alpha(a)b)=a\beta(b),
\end{equation*}
and thus is an example of Exel's transfer operator \cite{E}. Because we have $\alpha\beta(a)=\alpha(I)a\alpha(I)$ it is a complete transfer operator of \cite{BL} and in particular a non-degenerate transfer operator of \cite{E}, i.e. it satisfies $\alpha(\beta(I))=\alpha(I)$. Also notice that if the algebra $A$ is commutative, or more generally if $\alpha(I)$ is in the center of $A$, then $\beta$ is also an endomorphism:
\begin{equation*}
\beta(a)\beta(b)=\mathcal{J}^*a\mathcal{J}\mathcal{J}^*b\mathcal{J}=\mathcal{J}^*a\alpha(I)b\mathcal{J}=\mathcal{J}^*ab\alpha(I)\mathcal{J}=\mathcal{J}^*ab\mathcal{J}\mathcal{J}^*\mathcal{J}=\mathcal{J}^*ab\mathcal{J}=\beta(ab).
\end{equation*}

Clearly $\alpha(I)$ is a projection and notice that $\textrm{Ran}(\alpha)$, the range of $\alpha$, is  a hereditary subalgebra of $A$ and we have 
$$\textrm{Ran}(\alpha)=\alpha(I)A\alpha(I).$$ 

\subsection{Coefficient algebras and crossed products}
Suppose that $A$ is a unital C$^*$-algebra and $\alpha$ is an endomorphism of $A$. Stacey in \cite{St} (and Murphy in \cite{Mu}) defines the crossed product $A\rtimes_\alpha\N$ as the universal unital C$^*$-algebra generated by $a$'s in $A$, and an isometry ${\mathcal{J}}$,   subject to the relation:
\begin{equation*}
\alpha(a)= {\mathcal{J}}a{\mathcal{J}}^*, \ a \in A.
\end{equation*}

Given a unital C$^*$-algebra $A$, an endomorphism $\alpha : A \to A$, and a transfer operator $\beta$, the Exel’s crossed product is the universal C$^*$-algebra generated by a copy of $A$ and element ${\mathcal{J}}$ subject to the relations:
\begin{enumerate}
\item ${\mathcal{J}}a=\alpha(a){\mathcal{J}}$
\item ${\mathcal{J}}^*a{\mathcal{J}}=\beta(a)$
\item (Redundancy Condition) If for $a\in\overline{A\alpha(A)A}$ and $k\in\overline{A\mathcal{J}\mathcal{J}^*A}$ we have $ab{\mathcal{J}}=kb{\mathcal{J}}$ for all $b\in A$ then $a=k$.
\end{enumerate}

Interestingly, see for example the appendix of \cite{HKMP2}, if $\alpha$ is a monomorphism with hereditary range then both seemingly quite different concepts of crossed products by endomorphisms coincide if we take $\beta$ to be the transfer operator given by the equation:
\begin{equation*}
\beta(a):=\alpha^{-1}(\alpha(I)a\alpha(I)).
\end{equation*}

Suppose again that $A$ a coefficient algebra of the C$^*$-algebra $A_{\mathcal{J}}$. Using the above commutation relations one can arrange the polynomials in ${\mathcal{J}}$, ${\mathcal{J}}^*$ and $a$ so that the $a$'s are in front of the powers of ${\mathcal{J}}$ and the powers of ${\mathcal{J}}^*$ are in front of the $a$'s. This leads to the following observation (\cite{LO}, Proposition 2.3):
the vector space  $\mathcal{F}(A,{\mathcal{J}})$ consisting of finite sums
\begin{equation}\label{FourierV}
x=\sum_{n\ge0}a_n {\mathcal{J}}^n + \sum_{n<0}({\mathcal{J}}^*)^{-n} a_n\,,
\end{equation}
where $a_n\in A$, is a dense $*$-subalgebra of $A_{\mathcal{J}}$.
Additionally, using ${\mathcal{J}}^n = {\mathcal{J}}^n({\mathcal{J}}^*)^n{\mathcal{J}}^n$ and its adjoint
one can always choose coefficients $a_n$ such that
\begin{equation*}
a_n=a_n{\mathcal{J}}^n({\mathcal{J}}^*)^n ,\ n\geq 0 \textrm{ and } a_n={\mathcal{J}}^{-n}({\mathcal{J}}^*)^{-n}a_n,\ n<0.
\end{equation*}

We say that $A_{\mathcal{J}}$ satisfies the O'Donovan condition if for every $x\in\mathcal{F}(A,{\mathcal{J}})$ we have:
\begin{equation*}
||a_0||\leq ||x||.
\end{equation*}
A key result (\cite{LO} Theorem 2.7) says that if $A_{\mathcal{J}}$ satisfies the O'Donovan condition then the coefficients $a_n$ above are unique and, more importantly, $A_{\mathcal{J}}$ is isomorphic to the crossed product C$^*$-algebra $A\rtimes_\alpha\N$, see \cite{boyd1993faithful} and the appendix of \cite{HKMP2}.

If the algebra $A$ is commutative we get a stronger relation:
\begin{equation*}
a{\mathcal{J}} =  {\mathcal{J}}\beta(a).
\end{equation*}
Also, the $*$-subalgebra $\mathcal{F}(A,{\mathcal{J}})$ can be described as consisting of finite sums
\begin{equation*}
x=\sum_{n\ge0}{\mathcal{J}}^n b_n + \sum_{n<0}b_n ({\mathcal{J}}^*)^{-n}\,
\end{equation*}
with $b_n\in A$. Under the O'Donovan condition, the coefficients $b_n$ above are again unique.

\subsection{Gauge action}
In this and the following subsection we consider additional structure, not systematically explored in the existing literature, which we call a gauge action. For $\theta\in\R/\Z$, assume that there is a one-parameter family $\{\mathcal{U}_\theta\}$ of unitary operators on $\hcal$ satisfying the following commutation properties:
\begin{equation*}
\mathcal{U}_\theta\,a = a\, \mathcal{U}_\theta\,,\quad \mathcal{U}_\theta \mathcal{J} = e^{2\pi i\theta}\mathcal{J} \mathcal{U}_\theta\,
\end{equation*}
for all $a\in A$.
We define a one-parameter family $\{\rho_\theta\}$ of automorphisms $\rho_\theta:B(\hcal)\to B(\hcal)$  by the formula
\begin{equation*}
\rho_\theta(a)=\mathcal{U}_\theta a \mathcal{U}_\theta^{-1}.
\end{equation*}
It follows that on generators of $A_{\mathcal{J}}$ we have 
\begin{equation}\label{rhogen}
\rho_\theta(\mathcal{J})=e^{2\pi i\theta}\mathcal{J}\,,\quad \rho_\theta(\mathcal{J}^*)=e^{-2\pi i\theta}\mathcal{J}^*\,,\quad\textrm{and}\quad \rho_\theta(a) = a,\ \ a\in A.
\end{equation}
Clearly the map $\theta\mapsto \rho_\theta(x)$ is continuous for $x\in \mathcal{F}(A,{\mathcal{J}})$. It follows by an approximation argument that it is continuous for every $x\in A_{\mathcal{J}}$ 

As a consequence of the existence of a gauge action we have an expectation 
$$E: A_\mathcal{J}\to A_\mathcal{J}$$ 
given by the formula:
\begin{equation}\label{expect_formula}
E(x):=\int_0^1 \rho_\theta(x)\,d\theta.
\end{equation}
Then, for a finite sum given by equation \eqref{FourierV} we obtain
\begin{equation*}
a_n=E(x({\mathcal{J}}^*)^n) ,\ n\geq 0 \textrm{ and } a_n=E({\mathcal{J}}^{-n}x),\ n<0.
\end{equation*}
In particular, we have:
\begin{equation*}
||a_0||=||E(x)||\leq ||x||,
\end{equation*}
so the existence of a gauge action implies  the O'Donovan condition and so $A_{\mathcal{J}}\cong A\rtimes_\alpha\N$.

Consider now the range $\textrm{Ran}(E)$ of $E$. On one hand, for $x\in \mathcal{F}(A,{\mathcal{J}})$ we have $E(x)\in A$. It follows from the density of $\mathcal{F}(A,{\mathcal{J}})$ and from the continuity of $E$ that $\textrm{Ran}(E)=A$.
On the other hand, let $A_{\mathcal{J},\textrm{inv}}$ be the invariant subalgebra of $A_\mathcal{J}$ with respect to the automorphism $\rho_\theta$, that is
\begin{equation*}
A_{\mathcal{J},\textrm{inv}} = \{a\in A_\mathcal{J} : \rho_\theta(a)=a \textrm{ for all }\theta\in \R/\Z\}\,.
\end{equation*}
By the assumptions, we have $A\subseteq A_{\mathcal{J},\textrm{inv}}$. But for any $x\in A_\mathcal{J}$ we have from the definition that $E(x)\in A_{\mathcal{J},\textrm{inv}}$ so 
\begin{equation*}
A=\textrm{Ran}(E)=A_{\mathcal{J},\textrm{inv}}.
\end{equation*}
This observation is important for the topic we discuss next.

\subsection{Coefficient algebra extensions}
As before, let $A \subseteq B(\hcal)$ be a $*$-subalgebra containing the identity $I$ of $B(\hcal)$ and let ${\mathcal{J}} \in B(\hcal)$ be an isometry, ${\mathcal{J}}^*{\mathcal{J}}=I$. Assume that there exists a gauge action satisfying  equation \eqref{rhogen}, but we do not assume that $A$ is a coefficient algebra of $A_{\mathcal{J}}$. We want to consider enlargements of $A$ that form coefficients for $A_{\mathcal{J}}$. 

We call a $*$-subalgebra $B \subseteq B(\hcal)$ a coefficient algebra extension of $A$ if $A\subseteq B$, $A_{\mathcal{J}}=B_{\mathcal{J}}$ and $B$ is a coefficient algebra of $B_{\mathcal{J}}$. The collection of coefficient algebra extensions of $A$ is non-empty, since $A_{\mathcal{J}}$ itself is such an extension and it is closed under arbitrary intersections. It follows that there exists the {\it minimal coefficient algebra extension} of $A$.

Notice that, from the properties of $\rho_\theta$, for every $a\in A_{\mathcal{J},\textrm{inv}}$ we have: 
\begin{equation*}
\mathcal{J}^*a\mathcal{J}\in A_{\mathcal{J},\textrm{inv}}\quad\textrm{and}\quad \mathcal{J}a\mathcal{J}^*\in A_{\mathcal{J},\textrm{inv}}\,.
\end{equation*}
This means that $A_{\mathcal{J},\textrm{inv}}$ is a coefficient algebra extension of $A$. In particular, we have the following general result, which allows us to realize each of the shift algebras of the previous section as a crossed product by a monomorphism with a hereditary range.
\begin{prop} \label{CPProp}
There is an isomorphism of C$^*$-algebras:
\begin{equation*}
A_\mathcal{J} \cong A_{\mathcal{J},\textrm{inv}}\rtimes_{{\alpha}_\mathcal{J}}\N \,,
\end{equation*}
where ${{\alpha}_\mathcal{J}}:A_{\mathcal{J},\textrm{inv}}\to A_{\mathcal{J},\textrm{inv}}$ is an endomorphism, in fact a monomorphism with hereditary range,  given by
\begin{equation*}
{{\alpha}_\mathcal{J}}(a):=\mathcal{J}a\mathcal{J}^*.
\end{equation*}
Moreover, $A_{\mathcal{J},\textrm{inv}}$ is the minimal coefficient algebra extension of $A$.
\end{prop}
\begin{proof} 
We already mentioned that the existence of the gauge action implies the O'Donovan condition and thus $A_\mathcal{J} \cong A_{\mathcal{J},\textrm{inv}}\rtimes_{{\alpha}_\mathcal{J}}\N$. (Because ${{\alpha}_\mathcal{J}}$ is a monomorphism with hereditary range, all major definitions of the crossed product by endomorphism coincide in this case.)

Let $B$ be the minimal coefficient algebra extension of $A$. We observed above that the coefficient algebras coincides with the invariant subalgebras:
\begin{equation*}
B=\textrm{Ran}(E)=B_{\mathcal{J},\textrm{inv}}=A_{\mathcal{J},\textrm{inv}}
\end{equation*}
as claimed.
\end{proof}

\section{The Bunce-Deddens Shift}
\subsection{The Algebra $A_U$}
In this section we consider the first shift algebra 
\begin{equation*}
A_U=C^*(U,M_f:f\in C(\Z_s)),
\end{equation*}
where the shift $U$ is given on the basis elements of $H=\ell^2(\mathcal{V})$ by $UE_{(n,x)}= E_{(n+1,\,x+1)}$. The corresponding maps on $A_U$ are:
\begin{equation*}
\alpha_U(x):=U x U^* \in A_U,
\textrm{ and } \beta_U(x):=U^* x U \in A_U,
\end{equation*}
for $x\in A_U$.

We  define the maps $\mathfrak{a}_U$ and $\mathfrak{b}_U$ on $C(\Z_s)$ as follows:
\begin{equation*}
\mathfrak{a}_Uf(x) = f(x-1)\quad\textrm{and}\quad \mathfrak{b}_Uf(x) = f(x+1)\,.
\end{equation*}
Both are automorphisms on $C(\Z_s)$ and we have:
\begin{equation*} 
\mathfrak{a}_U(I)=\mathfrak{b}_U(I)=1 \textrm{\ and\ } \mathfrak{a}_U\circ\mathfrak{b}_U=\mathfrak{b}_U\circ\mathfrak{a}_U=1\,.
\end{equation*}
The following observation relates the maps $\mathfrak{a}_U$ and $\mathfrak{b}_U$ with $U$ and $M_f$.

\begin{lem}\label{A_U_relations}
For any $f\in C(\Z_s)$, the operators $U$ and $M_f$ satisfy the following relations:
\begin{equation*}
M_fU=UM_{\mathfrak{b}_Uf}\quad\textrm{and}\quad UM_fU^*=M_{\mathfrak{a}_Uf}UU^*\,.
\end{equation*}
\end{lem}

\begin{proof}
On the basis elements of $H$ we have
\begin{equation*}
M_fUE_{(n,x)} = M_fE_{(n+1,\,x+1)} = f(x+1)E_{(n+1,\,x+1)}
\end{equation*}
and 
\begin{equation*}
U M_{\mathfrak{b}_U f}E_{(n,x)} = U \mathfrak{b}_U f(x) E_{(n,x)} = f(x+1)UE_{(n,x)} = f(x+1)E_{(n+1,\,x+1)}.
\end{equation*}
This shows the first relation.  
We also have that
\begin{equation*}
\begin{aligned}
UM_fU^*E_{(n,x)} &= UM_f\left\{
\begin{aligned}
&E_{(n-1,\,x-1)} &&\textrm{ if }n\ge1, 0<x\le s^{n-1} \\
&0 &&\textrm{ if }n=0\textrm{ or }x=0\textrm{ or }s^{n-1}<x<s^n
\end{aligned}\right. \\
&=f(x-1)\,U\left\{
\begin{aligned}
&E_{(n-1,\,x-1)} &&\textrm{ if }n\ge1, 0<x\le s^{n-1} \\
&0 &&\textrm{ if }n=0\textrm{ or }x=0\textrm{ or }s^{n-1}<x<s^n
\end{aligned}\right. \\
&=M_{\mathfrak{a}_Uf}\,UU^*E_{(n,x)}\,.
\end{aligned}
\end{equation*}
\end{proof}

From Lemma \ref{A_U_relations}, and the fact that $U$ is an isometry, it follows additionally that 
\begin{equation}\label{A_U_relations2}
UM_f=M_{\mathfrak{a}_Uf}U \quad\textrm{and}\quad U^*M_fU = M_{\mathfrak{b}_U f}.
\end{equation}
 
\subsection{The Coefficient Algebra of $A_U$}

The main goal is to identify $A_U$ as a crossed product by an endomorphism. Notice however that  $UC(\Z_s)U^*$ is not generally contained in $C(\Z_s)$ by Lemma \ref{A_U_relations} and so $C(\Z_s)$ is not a coefficient algebra of  $A_U$. Crossed products by endomorphisms work best for coefficient algebras and we show below how to enlarge $C(\Z_s)$ to obtain such a coefficient algebra.

Define the mutually orthogonal projections $\{P_n\}$  by 
\begin{equation*} 
P_0=I-UU^*\textrm{ and } P_n = U^nP_0(U^*)^{n}.
\end{equation*} 
Below we show that the C$^*$-algebra
\begin{equation*}
B_U:=C^*(M_f,P_n:f\in C(\Z_s),n\in\Z_{\geq 0})
\end{equation*}
is the smallest coefficient algebra of $A_U$ containing $C(\Z_s)$.
\begin{prop}\label{BUProp}
The algebra $B_U$ is commutative and it is the minimal coefficient algebra extension of $C(\Z_s)$. Thus, we can identify $B_U$ with the invariant algebra 
\begin{equation*}
A_{U,\textrm{inv}} = \{a\in A_U : \rho_\theta(a)=a \textrm{ for all }\theta\in \R/\Z\}.
\end{equation*}
\end{prop}
\begin{proof}
 From the relations of Lemma \ref{A_U_relations}, equations \eqref{A_U_relations2} and the fact that $\mathfrak{a}_U$ and $\mathfrak{b}_U$ are inverses of each other, it follows that $P_n$ commutes with $M_f$ for every $n$ and $f$. This demonstrates that $B_U$ is commutative.

Using formulas
\begin{equation*}
\beta_U(P_n)=U^*P_nU=P_{n-1} \textrm{ for } n\geq1, \ \beta_U(P_0)=0,\ \alpha_U(P_n)=UP_nU^*=P_{n+1}\textrm{ for } n\geq0
\end{equation*}
as well as
\begin{equation*}
U^*M_fU=M_{\mathfrak{b}_Uf},\ UM_fU^*=M_{\mathfrak{a}_Uf}(I-P_0),
\end{equation*}
and commutativity of $B_U$, we see that $\beta_U$ is a homomorphism and thus, $B_U$ is a coefficient algebra extension of $C(\Z_s)$.  Notice that, by definition, $P_n=\beta_U^n(I)- \beta_U^{n+1}(I)$ must be in any coefficient algebra extension of $C(\Z_s)$ so $B_U$ must be the smallest such extension. Finally, we noticed in the previous section that in general the gauge invariant subalgebra is the smallest coefficient algebra extension.
\end{proof}

\subsection{The Structure of the Coefficient Algebra of $A_U$}

We can also describe the smallest coefficient algebra extension of $C(\Z_s)$ more explicitly. In fact, we demonstrate below that the algebra $A_{U,\textrm{inv}}$ is isomorphic to the tensor product $c(\Z_{\ge0})\otimes C(\Z_s)$, where $c(\Z_{\ge0})$ is the space of convergent sequences.  

We start with a quick summary of relevant information about tensor products of C$^*$-algebras, see for example \cite{Murphybook}.
\begin{itemize}
\item All Abelian C$^*$-algebras are nuclear.
\item If $A$ is nuclear then for any C$^*$-algebra $B$ there is only one C$^*$-norm on the algebraic tensor product of $A$ and $B$. The corresponding completion is simply denoted by $A\otimes B$ below.
\item If $X$ and $Y$ are compact Hausdorff spaces and  $B$ is any C$^*$-algebra then we have $C(X)\otimes C(Y)\cong C(X\times Y)$ and $C(X)\otimes B\cong C(X,B)$.
\end{itemize}
Thus, we can view $c(\Z_{\ge0})\otimes C(\Z_s)$ as the space of uniformly convergent sequences of continuous functions on $\Z_s$:
\begin{equation*}
c(\Z_{\ge0})\otimes C(\Z_s)\cong c(\Z_{\ge0}, C(\Z_s))\cong \{F=(f_n)_{n\in \Z_{\ge0}}: f_n\in C(\Z_s),\  \lim_{n\to\infty}f_n\textrm{ exists}\}\,.
\end{equation*}
We use the following notation for the uniform limit:
\begin{equation*}
C(\Z_s)\ni f_\infty:=\lim_{n\to\infty}f_n.
\end{equation*}

\begin{prop}\label{BUThm}
The map
\begin{equation}\label{ToeplitzU}
c(\Z_{\ge0}, C(\Z_s))\ni F\mapsto T_U(F):= \sum_{n=0}^{\infty}(M_{f_n}-M_{f_\infty})P_n+M_{f_\infty} \in B_U=A_{U,\textrm{inv}}
\end{equation}
 is an isomorphism of C$^*$-algebras.
\end{prop}

\begin{proof}
Notice that for block diagonal operators the norm is equal to the supremum of the norms of the blocks. It follows that
\begin{equation*}
\left\|\sum_{n=0}^{\infty}(M_{f_n}-M_{f_\infty})P_n\right\|\leq\sup_n||M_{f_n}-M_{f_\infty}||=||f_n-f_\infty||_\infty<\infty,
\end{equation*}
and so the map $F\mapsto T_U(F)$ is well defined.

Let $P^\perp$ be the projection onto the subspace of $H$ orthogonal to the ranges of all the $P_n$'s. Then we can write an alternative formula for $T_U(F)$ as follows:
\begin{equation*}
T_U(F):= \sum_{n=0}^{\infty}M_{f_n}P_n+M_{f_\infty}P^\perp.
\end{equation*}
Then, because all terms in the above formula commute and the projections are mutually orthogonal, we see that the map $F\mapsto T_U(F)$ is a C$^*$-homomorphism. Additionally, the formula in equation \eqref{ToeplitzU} implies that the generators of $B_U$ are in the range of $T_U$. Consequently, the range of $T_U$ must be equal to $B_U$.

To show that the map $F\mapsto T_U(F)$ is an isomorphism we compute the norm $||T_U(F)||$ using the observation above about block diagonal operators. This requires to compute the norm of $M_{f_n}$ restricted to the range of $P_n$.

A simple calculation using the formulas for $U$ and $U^*$ yields
\begin{equation*}
\textrm{ Ran }P_n=\textrm{span}\,\{ E_{(m+n, x+n)}: m=0,\, x=0 \textrm{ or }m\geq 1,\, s^{m-1}<x<s^m\}.
\end{equation*}
As a consequence, we have the following observation regarding the range of $P_n$:
\begin{equation*}\begin{aligned}
&\{x\in \Z_{\geq 0}:\text{there exists } m \text{ so that } E_{(m,x)}\in \text{Ran\ } P_n\}\\
&=\{ x\in \Z_{\geq 0}: x\geq n,\, x\ne n+s^m,\, m\in \Z_{\geq 0} \}.
\end{aligned}
\end{equation*}
Using this observation we have that
\begin{equation*}
\begin{aligned}
\|M_{f_n}|_{\textrm{Ran }P_n}\|&= 
\sup \left\{|f_n(x)|: x\text{\ such that there exists\ } m \text{\ so that\ } E_{(m,x)}\in \text{Ran\ } P_n\right\} \\
&= \sup_{x\in\Z_s}|f_n(x)|=||f_n||_\infty,
 \end{aligned}
\end{equation*}
since the set $\{x:\text{there exists\ } m \text{\ so that\ } E_{(m,x)}\in \text{Ran\ } P_n\}$, described explicitly above, is dense $\Z_s$. Thus, we have 
\begin{equation*}
||T_U(F)||=\sup_n||f_n||_\infty
\end{equation*}
since the norm of the last term, $M_{f_\infty}P^\perp$, is no more than the above supremum.
It follows that $B_U$ is isomorphic to $c(\Z_{\ge0})\otimes C(\Z_s)$.
\end{proof}

\subsection{$A_U$ as a Crossed Product} Now that we have identified a coefficient algebra of $A_U$, the invariant subalgebra $B_U$ isomorphic to $c(\Z_{\ge0})\otimes C(\Z_s)$, we can use the general results from Section 3 to describe $A_U$ as a crossed product by an endomorphism.

Consider the morphisms $\tilde{\alpha}_U,\tilde{\beta}_U$ on $c(\Z_{\ge0})\otimes C(\Z_s)$ given by
\begin{equation*}
\tilde{\alpha}_UF(n,x) = \left\{
\begin{aligned}
&F(n-1,\,x-1) &&\textrm{if }n\ge0 \\
&0 &&\textrm{if }n=0
\end{aligned}\right.\quad\textrm{and}\quad \tilde{\beta}_UF(n,x) = F(n+1,\,x+1)\,.
\end{equation*}
A simple calculation shows that we have the following relations: 
\begin{equation*}
\alpha_U(T_U(F))=U^*T_U(F)U=T_U(\tilde{\beta}_UF)\quad\textrm{and}\quad \beta_U(T_U(F))=UT_U(F)U^*=T_U(\tilde{\alpha}_UF)\,.
\end{equation*}

The following main result of this section is now a direct consequence of Proposition \ref{CPProp}, Proposition \ref{BUProp}, Proposition \ref{BUThm} and the above formula. Clearly, $\tilde{\alpha}_U$ is a monomorphism of $c(\Z_{\ge0})\otimes C(\Z_s)$ with hereditary range.

\begin{theo}\label{AUTheo}
We have the following isomorphism of C$^*$-algebras:
\begin{equation*}
A_U \cong (c(\Z_{\ge0})\otimes C(\Z_s))\rtimes_{\tilde{\alpha}_U}\N \,.
\end{equation*}
\end{theo} 

The significance of the above result is that the crossed product on the right-hand side is a universal object, defined in terms of generators and relations, and thus can be studied by considering its convenient, concrete representations.

\section{The Hensel Shift}
\subsection{The Algebra $A_V$}
In this section, we consider the algebra 
\begin{equation*} 
A_V=C^*(V,M_f: f\in C(\Z_s)),
\end{equation*}
where the Hensel shift $V$ is given on the basis elements of $H$ by $VE_{(n,x)} = E_{(n+1,\,sx)}$.
The corresponding maps on $A_V$ are:
\begin{equation*}
\alpha_V(x):=V x V^* \in A_V,
\textrm{ and } \beta_V(x):=V^* x V \in A_V,
\end{equation*}
for $x\in A_V$.
The analysis turns out to be quite similar the Bunce-Deddens shift and in particular the gauge invariant coefficient algebra is also Abelian.

Consider also the following maps: $\mathfrak{a}_V,\mathfrak{b}_V: C(\Z_s)\to C(\Z_s)$ given by
\begin{equation*}
\mathfrak{a}_Vf(x) = \left\{
\begin{aligned}
&f\left(\frac{x}{s}\right) &&\textrm{if }s|x\\
&0 &&\textrm{else}
\end{aligned}\right.\quad\textrm{and}\quad \mathfrak{b}_Vf(x)=f(sx)\,.
\end{equation*}
It is routine to check that both $\mathfrak{a}_V$ and $\mathfrak{b}_V$ are endomorphisms of $C(\Z_s)$.  We have that
\begin{equation*}
(\mathfrak{b}_V\circ \mathfrak{a}_V)f(x) = \mathfrak{b}_V(\mathfrak{a}_Vf(x))=\mathfrak{a}_Vf(sx)=f(x)\,,
\end{equation*}
and thus $\mathfrak{b}_V\circ\mathfrak{a}_V=I$.  Additionally, we have 
\begin{equation*}
\mathfrak{b}_V(I)=I \textrm{ and } (\mathfrak{a}_V\circ \mathfrak{b}_V)f=\mathfrak{a}_V(I)f.
\end{equation*}
Here
\begin{equation*}
\mathfrak{a}_V (I)(x) = \left\{
\begin{aligned}
&1&&\textrm{if }s|x \\
&0 &&\textrm{else}
\end{aligned}\right.
\end{equation*}
is the characteristic function of the subset of  $\Z_s$ of elements divisible by $s$. It follows from that formula that the range of $\mathfrak{a}_V$ is $\mathfrak{a}_V(I)C(\Z_s)$. Thus $\mathfrak{a}_V$ is a monomorphism with hereditary range.

\begin{lem}\label{A_V_relations}
Given $f\in C(\Z_s)$, then $V$ and $M_f$ satisfy the following relations:
\begin{equation*}
M_fV=VM_{\mathfrak{b}_Vf}\quad\textrm{and}\quad VM_fV^*=M_{\mathfrak{a}_Vf}(I-P_{(0,0)})\,,
\end{equation*}
where $P_{(0,0)}$ is the projection onto (the subspace generated by) the basis element $E_{(0,0)}$.
\end{lem}

\begin{proof}
We first apply $M_fV$ to the basis elements of $H$:
\begin{equation*}
M_fVE_{(n,x)} = M_fE_{(n+1,\,sx)} = f(sx)E_{(n+1,\,sx)}=VM_{\mathfrak{b}_Vf}E_{(n,x)}\,.
\end{equation*}
Similarly we have
\begin{equation*}
\begin{aligned}
VM_fV^*E_{(n,x)} &= VM_f\left\{
\begin{aligned}
&E_{(n-1,\,x/s)} &&\textrm{if }n\ge1, s|x \\
&0 &&\textrm{else}
\end{aligned}\right. = f\left(\frac{x}{s}\right)V\left\{
\begin{aligned}
&E_{(n-1,\,x/s)} &&\textrm{if }n\ge1, s|x \\
&0 &&\textrm{else}
\end{aligned}\right. \\
&= \left\{
\begin{aligned}
&f\left(\frac{x}{s}\right)E_{(n,x)} &&\textrm{if }n\ge1, s|x \\
&0 &&\textrm{else}
\end{aligned}\right. = M_{\mathfrak{a}_Vf}(I-P_{(0,0)})E_{(n,x)}\,.
\end{aligned}
\end{equation*}
\end{proof}
From Lemma \ref{A_V_relations}, and the fact that $V$ is an isometry, it follows, similarly to the Bunce-Deddens shift, that we have:
\begin{equation*}
VM_f=M_{\mathfrak{a}_Vf}V \quad\textrm{and}\quad V^*M_fV = M_{\mathfrak{b}_V f}.
\end{equation*} 

\subsection{The Coefficient Algebra of $A_V$}

We have the following observation: 
$$M_fP_{(0,0)}=f(0)P_{(0,0)},$$ 
and consequently Lemma \ref{A_V_relations} implies that
\begin{equation*}
P_{(0,0)}= M_{\mathfrak{a}_V(I)}-VV^*\,.
\end{equation*}
This shows that $P_{(0,0)}$ is an element of $A_V$.
If we define 
\begin{equation*}
P_{(n,0)} = V^nP_{(0,0)}(V^*)^{n}\textrm{ for }n\in\Z_{\ge0}\,,
\end{equation*}
we get $P_{(n,0)}\in A_V$ for all $n\in\Z_{\ge0}$. Additionally, $P_{(n,0)}$ is the one-dimensional projection onto (the subspace generated by) the basis element $E_{(n,0)}$.

Below we show that the C$^*$-algebra $B_V$ generated by $M_f$'s and $P_{(n,0)}$:
\begin{equation*}
B_V:=C^*(M_f,P_{(n,0)}:f\in C(\Z_s),n\in\Z_{\geq 0})
\end{equation*}
is the smallest coefficient algebra of $A_V$ containing $C(\Z_s)$. 

\begin{prop}\label{BVProp}
The algebra $B_V$ is commutative, it is the minimal coefficient algebra extension of $C(\Z_s)$, and the algebra $A_{V,\textrm{inv}}$ is equal to $B_V$. 
\end{prop}

\begin{proof}
We already established in Proposition \ref{CPProp} that the gauge invariant subalgebra $A_{V,\textrm{inv}}$ is the minimal coefficient algebra extension of $C(\Z_s)$.

 It easily follows from the above definitions that we have the following relation:
\begin{equation*}
M_fP_{(n,0)}=f(0)P_{(n,0)}\,.
\end{equation*}
Consequently, $B_V$ is commutative.

Next, we want to show that $B_V$ is a coefficient algebra for $A_V$, i.e. it is closed under the actions of $\alpha_V$ and $\beta_V$. Since $B_V$ is commutative, $\beta_V$ restricted to $B_V$ is a homomorphism. Thus, it is enough to inspect the actions of $\alpha_V$ and $\beta_V$ on generators of $B_V$.
 The formulas
\begin{equation*}
\alpha_V(P_{(n,0)})=VP_{(n,0)}V^*= P_{(n+1,0)}\,,\quad \beta_V(P_{(n,0)})=V^*P_{(n,0)}V=P_{(n-1,0)}
\end{equation*}
(taking $P_{(-1,0)}=0$) and
\begin{equation*}
V^*M_fV = M_{\mathfrak{b}_V f}\textrm{ and } VM_fV^*=M_{\mathfrak{a}_Vf}(I-P_{(0,0)}),
\end{equation*}
clearly establish that $B_V$ is a coefficient algebra extension of $C(\Z_s)$.  

Notice that from the definitions, the generators of $B_V$ are gauge invariant and so we have the inclusion $B_V\subseteq A_{V,\textrm{inv}}$. But since  $A_{V,\textrm{inv}}$ is the minimal coefficient algebra, $B_V$ must be equal to $A_{V,\textrm{inv}}$. 
\end{proof}

\subsection{The Structure of the Coefficient Algebra of $A_V$}
Since $B_V$ is a unital, commutative C$^*$-algebra it must be of the form $B_V\cong C(X_V)$ for some compact Hausdorff  space $X_V$. The purpose here is to describe such a space.

Let $X_V$ be the following closed subspace of the compact topological space $\Z_s\times [0,1]$:
\begin{equation*}
X_V=\{(x,0):x\in\Z_s\}\cup\left\{\left(0,\frac{1}{n+1}\right):n\in\Z_{\geq 0}\right\}.
\end{equation*}

Since $X_V$ is a union of two sets, any continuous function $F$ on $X_V$ consists of two functions: a continuous function $f$ on $\Z_s\cong \{(x,0):x\in\Z_s\}$ and a sequence $(x_n)_{n=0}^\infty$ of values of $F$ on the discrete set $\{\left(0,\frac{1}{n+1}\right):n\in\Z_{\geq 0}\}$, while the continuity of $F$ at $(0,0)$ requires $\lim\limits_{n\to\infty}x_n=f(0)$. Consequently, we have:
\begin{equation*}
C(X_V)\cong\{F=(f, (x_n)_{n=0}^\infty): f\in C(\Z_s), \lim\limits_{n\to\infty}x_n=f(0)\}
\end{equation*}

\begin{prop}\label{BVThm}
The map
\begin{equation}\label{ToeplitzV}
C(X_V)\ni F\mapsto T_V(F):= \sum_{n=0}^{\infty}(x_n-f(0))P_{(n,0)}+M_{f} \in B_V= A_{V,\textrm{inv}}
\end{equation}
 is an isomorphism of C$^*$-algebras.
\end{prop}
\begin{proof}
Let $P_*^\perp$ be the projection onto the subspace of $H$ orthogonal to the ranges of all the $P_{(n,0)}$'s. Then we can write an alternative formula for $T_V(F)$ as follows:
\begin{equation*}
T_V(F):= \sum_{n=0}^{\infty}x_nP_{(n,0)}+M_fP_*^\perp.
\end{equation*}
It is easily seen from those formulas that the map $F\mapsto T_V(F)$ is a C$^*$-homomorphism and that the generators of $B_V$ are in the range of $T_V$, implying that the range of $T_V$ must be equal to $B_V$.

As in the proof of Proposition \ref{BUThm}  we compute the norm $||T_V(F)||$ using the fact about the norms of block diagonal operators. This requires to compute the norm of $M_{f}$ restricted to the range of $P_*^\perp$, that range is simply spanned by those $ E_{(m,x)}$ with nonzero $x$.
Consequently, we have that
\begin{equation*}
\begin{aligned}
\|M_{f}|_{\textrm{Ran }P_*^\perp}\|&= 
\sup \left\{|f(x)|: x\text{\ such that there exists\ } m \text{\ so that\ } E_{(m,x)}\in \text{Ran\ } P_*^\perp\right\} \\
&= \sup_{0<x\in\Z}|f(x)|=||f||_\infty,
 \end{aligned}
\end{equation*}
since the set $\{x\in\Z:x>0\}$ is dense $\Z_s$. Hence,
\begin{equation*}
||T_V(F)||=||F||_\infty,
\end{equation*}
establishing the desired isomorphism.
\end{proof}

\subsection{$A_V$ as a Crossed Product}

Consider the morphisms $\tilde{\alpha}_V,\tilde{\beta}_V$ on $C(X_V)$ given by
\begin{equation*}
\tilde{\alpha}_V(f,(x_n)_{n=0}^\infty)= (\mathfrak{a}_V f,(x_{n-1})_{n=0}^\infty)\quad\textrm{and}\quad\tilde{\beta}_V(f,(x_n)_{n=0}^\infty) = (\mathfrak{b}_V f,(x_{n+1})_{n=0}^\infty),
\end{equation*}
taking $x_{-1}=0$.
A simple calculation shows that we have the following relations: 
\begin{equation}\label{Vtilde}
V^*T_V(F)V=T_V(\tilde{\beta}_VF)\quad\textrm{and}\quad VT_V(F)V^*=T_V(\tilde{\alpha}_VF)\,.
\end{equation}

Similarly to the previous section we obtain the following main result of this section as a direct consequence of Proposition \ref{CPProp}, Proposition \ref{BVThm}, and equations \eqref{Vtilde}.
\begin{theo}\label{AVTheo}
We have the following isomorphism of C$^*$-algebras:
\begin{equation*}
A_V \cong C(X_V)\rtimes_{\tilde{\alpha}_V}\N \,.
\end{equation*}
\end{theo}

\section{The Bernoulli Shift}
\subsection{The Algebra $A_S$}
Recall that the Bernoulli Shift  $S:H\to H$ is defined as:
\begin{equation*}
SE_{(n,x)}=\frac{1}{s}\sum_{j=0}^{s-1}E_{(n+1,\,sx+j)}\,.
\end{equation*}
In this section we study the algebra
\begin{equation*}
A_S=C^*(S,M_f: f\in C(\Z_s)).
\end{equation*} 
The goal, as before, is to identify the gauge-invariant subalgebra and write the algebra $A_S$ as a crossed product of a coefficient algebra by an endomorphism. It turns out that the gauge-invariant subalgebra is non-commutative and the analysis of the Bernoulli shift has a slightly different flavor than the previous two shifts.

Like in the other cases we have maps $\mathfrak{a}_S,\mathfrak{b}_S:C(\Z_s)\to C(\Z_s)$ defined by
\begin{equation*}
\mathfrak{a}_Sf(x) = f\left(\frac{x-x\textrm{ mod }s}{s}\right)\quad\textrm{and}\quad \mathfrak{b}_Sf(x) = \frac{1}{s}\sum_{j=0}^{s-1}f(sx+j)\,.
\end{equation*}
It is clear that $\mathfrak{a}_S$ is an endomorphism but $\mathfrak{b}_S$ is not a homomorphism as it was before, however  $\mathfrak{b}_S$ is positive and linear.  We also have that $\mathfrak{b}_S(I)=I=\mathfrak{a}_S(I)$.  It is also straightforward to see that $\mathfrak{a}_S$ has trivial kernel.

Consider the following computation:
\begin{equation*}
(\mathfrak{b}_S\circ\mathfrak{a}_S)f(x)=\mathfrak{b}_S(\mathfrak{a}_Sf(x))=\frac{1}{s}\sum_{j=0}^{s-1}\mathfrak{a}_Sf(sx+j)=\frac{1}{s}\sum_{j=0}^{s-1}f(x)=f(x)\,.
\end{equation*}
Consequently this implies that $\mathfrak{b}_S\circ\mathfrak{a}_S=1$ and that $\textrm{Ran }\mathfrak{b}_S=C(\Z_s)$.

\begin{lem}\label{A_S_relations}
Given $f\in C(\Z_s)$, $S$ and $M_f$ satisfy the following relations:
\begin{equation*}
S^*M_fS=M_{\mathfrak{b}_Sf}\quad\textrm{and}\quad SM_f=M_{\mathfrak{a}_Sf}S\,.
\end{equation*}
\end{lem}

\begin{proof}
We check the two relations on the basis elements $\{E_{(n,x)}\}$.  Consider the following:
\begin{equation*}
S^*M_fSE_{(n,x)}=\frac{1}{s}S^*M_f\sum_{j=0}^{s-1}E_{(n+1,\,sx+j)}=\frac{1}{s}S^*\sum_{j=0}^{s-1}f(sx+j)E_{(n+1,\,sx+j)}=M_{\mathfrak{b}_Sf}E_{(n,x)}
\end{equation*}
proving the first relation.  For second relation we compare the following two calculations:
\begin{equation*}
SM_fE_{(n,x)} = f(x)\frac{1}{s}\sum_{j=0}^{s-1}E_{(n+1,\,sx+j)}
\end{equation*}
while
\begin{equation*}
M_{\mathfrak{a}_Sf}SE_{(n,x)} = M_{\mathfrak{a}_Sf}E_{(n+1,\,sx+j)}=\frac{1}{s}\sum_{j=0}^{s-1}\mathfrak{a}_Sf(sx+j)E_{(n+1,\,sx+j)}=f(x)\frac{1}{s}\sum_{j=0}^{s-1}E_{(n+1,\,sx+j)}\,.
\end{equation*}
These two show the second relation completing the proof.
\end{proof}

\subsection{Relations with Cuntz-Toeplitz Algebras}
The goal of this subsection is to identify $A_S$ with a well-understood C$^*$-algebra that is easier to describe.  For $j=0,1,\ldots, s-1$, consider the characteristic functions $\chi_j$ given by the following formula for any $x\in\Z_s$:
\begin{equation*}
\chi_j(x) =\left\{
\begin{aligned}
&1 &&\textrm{ if }x \textrm{ mod }s=j \\
&0 &&\textrm{else.}
\end{aligned}\right.
\end{equation*}
This is a characteristic function for the ball with center $j$ and radius $s^{-1}$.
Notice that for each $j$, $\chi_j$ is locally constant and hence is a continuous function on $\Z_s$.  Thus for each $j$, $\chi_j$ defines a map $M_{\chi_j}:H\to H$ and we can define operators $S_j:H\to H$ for each $j=0,\ldots s-1$ by 
\begin{equation*} 
S_j=sM_{\chi_j}S.
\end{equation*}  
Let $C_S$ be the C$^*$-algebra generated by the $S_j$'s, that is, 
\begin{equation*}
C_S:=C^*(S_j:j=0,\ldots,s-1).
\end{equation*}
We have the following key observation.

\begin{prop}
We have the following equality of C$^*$-algebras:
\begin{equation*}
A_S = C_S.
\end{equation*}
\end{prop}

\begin{proof}
For each $j=0,\ldots, s-1$, we have by definition that $S_j\in A_S$, and thus we have $C_S\subseteq A_S$.  To show the reverse containment, we show that the generators of $A_S$ are in $C_S$.  

Note that
\begin{equation}\label{S_j defn}
S_jE_{(n,x)} = M_{\chi_j}\sum_{l=0}^{s-1}E_{(n+1,\,sx+l)}=E_{(n+1,\,sx+j)}\,,
\end{equation}
and so it follows that
\begin{equation*}
S=\frac{1}{s}\sum_{j=0}^{s-1}S_j\,.
\end{equation*}
Consequently we have that $S\in C_S$.  Direct calculations show that $S_j^*S_k=\delta_{jk}I$ and that
\begin{equation*}
\begin{aligned}
S_jS_j^*E_{(n,x)} 
&=S_j\left\{
\begin{aligned}
&E_{(n-1,\,(x-x\textrm{ mod }s)/s)} &&\textrm{ if }n\ge1, x\textrm{ mod }s=j \\
&0 &&\textrm{else}
\end{aligned}\right.\\
&= \left\{
\begin{aligned}
&E_{(n,x)} &&\textrm{ if }n\ge1, x\textrm{ mod }s=j \\
&0 &&\textrm{else}
\end{aligned}\right. \\
&= (I-P_{(0,0)})M_{\chi_j}E_{(n,x)}\,
\end{aligned}
\end{equation*}
where, as before, $P_{(0,0)}$ is the projection onto (the subspace generated by) the basis element $E_{(0,0)}$.  It follows that we have:
\begin{equation*}
\sum_{j=0}^{s-1}S_jS_j^*=(I-P_{(0,0)})\sum_{j=0}^{s-1}M_{\chi_j}=I-P_{(0,0)}
\end{equation*}
and hence $P_{(0,0)}\in C_S$.  Notice that for $j=0,\ldots,s-1$ we have
\begin{equation*}
M_{\chi_j}= (I-P_{(0,0)})M_{\chi_j} + P_{(0,0)}M_{\chi_j} = S_jS_j^*+\delta_{j0}P_{(0,0)}
\end{equation*}
and so $M_{\chi_j}\in C_S$.  Also notice that for each $k=0,\ldots,s-1$,
\begin{equation*}
\begin{aligned}
S_kM_{\chi_j}S_k^*E_{(n,x)}&=\left\{
\begin{aligned}
&\chi_j\left(\frac{x-k}{s}\right)E_{(n,x)} &&\textrm{ if }n\ge1, x\textrm{ mod }s=k \\
&0 &&\textrm{else}
\end{aligned}\right.\\
&=\left\{
\begin{aligned}
&E_{(n,x)} &&\textrm{ if }n\ge1, x\textrm{ mod }s=k,\frac{x-k}{s}\textrm{ mod }s=j \\
&0 &&\textrm{else}
\end{aligned}\right.
\end{aligned}
\end{equation*}
which is just a multiplication operator by the characteristic function for the ball centered at $k+sj$ with radius $s^{-2}$. Consequently that multiplication operator is in $C_S$.  Repeating this argument yields that for any ball, multiplication by the corresponding characteristic function is in $C_S$.

Since any locally constant function is a finite linear combination of characteristic functions we have that multiplication by locally constant functions also belong to $C_S$.  Since the locally constant functions on $\Z_s$ are dense in $C(\Z_s)$, for each $f\in C(\Z_s)$ we get $M_f\in C_S$.  This completes the proof.
\end{proof}

While the above proposition shows that the algebras are the same, it also yields two key relations between the $S_j$'s, namely 
\begin{equation}\label{A_S_cuntz_rel}
S_j^*S_k=\delta_{jk}I \text{ for all } j,k=0,\ldots,s-1 \quad \text{and}\quad\sum_{j=0}^{s-1}S_jS_j^*=I-P_{(0,0)}\,.
\end{equation}
This identifies $A_S$ with the Cuntz-Toeplitz algebra, see below.

It is clear that $P_{(0,0)}$ is a diagonal compact operator, and in fact, the following proposition shows that all compact operators on $H$ are in $C_S$.  
\begin{prop}\label{compact_A_S}
The algebra $\mathcal{K}(H)$ of compact operators on $H$ is an ideal of $A_S$.
\end{prop}
\begin{proof}
For $n\in \Z_{\ge0}$, $0 \le x < s^n$, expand $x$ as
\begin{equation*}
x = x_0 + x_1 s + x_2 s^2 + \cdots + x_{n-1} s^{n-1},
\end{equation*}
with the coefficients satisfying $0\leq x_k<s$, for $k=0,\ldots,n-1$.
Define the following operator on $H$ by
\begin{equation*}
S_{(n,x)} = S_{x_0}S_{x_1}\cdots S_{x_{n-1}}
\end{equation*} 
where $S_{x_k}$ is given by equation \eqref{S_j defn} for $k=0,\ldots,n-1$.  Clearly $S_{(n,x)}, S_{(n,x)}^*$ are in $A_S$.  Define the following elements of $A_S$ for $0\le x < s^n, 0\le y < s^m$:
\begin{equation*}
P_{(n,x),(m,y)} = S_{(n,x)} P_{(0,0)}S_{(m,y)}^*.
\end{equation*}
By direct calculation we have
that for $m,n\in\Z_{\ge0}$ and $0\le x<s^n$, and $0\le x<s^m$
\begin{equation*}
\{P_{(n,x),(m,y)}\}_{n,m,x,y}
\end{equation*}
form a system of matrix units for the standard basis $\{E_{(n,x)}\}$ of $H$.  These matrix units generate $\kcal(H)$.
\end{proof}

\subsection{Cuntz and Cuntz-Toeplitz Algebras}
We recall here basic definitions and properties of Cuntz and Cuntz-Toeplitz algebras from \cite{Cu3}, following the detailed treatment in \cite{KD}.

Let $\mathcal{O}_s$ be the Cuntz algebra indexed by $s\ge2$.  This algebra is the universal C$^*$-algebra with generators $\{u_j\}_{j=0}^{s-1}$ and relations
\begin{equation}\label{Cuntzrelations}
u_j^*u_k=\delta_{jk}\quad\textrm{and}\quad \sum_{j=0}^{s-1}u_ju_j^*=1\,.
\end{equation}
The algebra $\mathcal{O}_s$ is simple which follows from the following strong property: for every non-zero $a\in \mathcal{O}_s$ there are $a_1, a_2\in \mathcal{O}_s$ such that
\begin{equation*}
a_1aa_2=I,
\end{equation*}
see Theorem V.5.6 of \cite{KD}.

As before, for $n\in \Z_{\ge0}$, $0 \le x < s^n$, expand $x$ as
\begin{equation}\label{xexpansion}
x = x_0 + x_1 s + x_2 s^2 + \cdots + x_{n-1} s^{n-1}.
\end{equation}
Define the following elements of $\mathcal{O}_s$:
\begin{equation*}
u_{(n,x)} = u_{x_0}u_{x_1}\cdots u_{x_{n-1}}.
\end{equation*} 
By Lemma V.4.1 of \cite{KD} any product of $u_{(n,x)}$'s  and  $u_{(m,y)}^*$'s has a reduced expression with no $u_{(n,x)}$'s to the right of any $u_{(m,y)}^*$'s.
Consequently, the *-subalgebra Pol$(\mathcal{O}_s)$ of finite linear combinations of products $u_{(n,x)}u_{(m,y)}^*$ is dense in $\mathcal{O}_s$.

A Cuntz-Toeplitz algebra $\mathcal{TO}_s$ indexed by $s\ge2$ is the universal C$^*$-algebra generated by $s$ isometries $s_i$, $i=0,\ldots,s-1$ such that
\begin{equation*}
s_j^*s_k=\delta_{jk}I \text{ for all } j,k=0,\ldots,s-1\,.
\end{equation*}
For $(n,x)$ as in equation \eqref{xexpansion} we define the following elements of $\mathcal{TO}_s$:
\begin{equation*}
s_{(n,x)} = s_{x_0}s_{x_1}\cdots s_{x_{n-1}}.
\end{equation*} 
Similarly, the *-subalgebra Pol$(\mathcal{TO}_s)$ of finite linear combinations of products $s_{(n,x)}s_{(m,y)}^*$ is dense in $\mathcal{TO}_s$.

To understand the irreducible representations of a  Cuntz-Toeplitz algebra, consider two cases of representations: $\sum_{j=0}^{s-1}s_js_j^*=I$ and  $\sum_{j=0}^{s-1}s_js_j^*<I$. In the first case, a representation of $\mathcal{TO}_s$ is a representation of $\mathcal{O}_s$ which, by simplicity of $\mathcal{O}_s$, is unique up to isomorphism. In the second case the projection \begin{equation*}
P:=I-\sum_{j=0}^{s-1}s_js_j^*
\end{equation*} 
is non-zero. If $e_{(0,0)}$ is a unit vector in the range of $P$, then $Pe_{(0,0)}=e_{(0,0)}$ and so we have $s_j^*e_{(0,0)}=0$ for all $j$. It is easy to verify that the vectors $e_{(n,x)}:=s_{(n,x)}e_{(0,0)}$ form an orthonormal set and their span is invariant for such a representation of $\mathcal{TO}_s$ and thus they span the whole Hilbert space by irreducibility. Such a representation is isomorphic to the representation of $A_S$ in $B(H)$ by correspondence $e_{(n,x)}\mapsto E_{(n,x)}$ which sends $s_j$ to $S_j$ for all $j$. Consequently, there are only two irreducible representations of $\mathcal{TO}_s$. Lemma V.5.2 of \cite{KD} says that the ideal $\langle  P\rangle$ generated by the projection $P$ is isomorphic to the algebra of compact operators and the factor algebra $\mathcal{TO}_s/\langle  P\rangle$ is isomorphic to the Cuntz algebra $\mathcal{O}_s$. It follows that the representation of $A_S$ in $B(H)$ is a defining representation for the Cuntz-Toeplitz algebra.

To summarize, our key relations in equation (\ref{A_S_cuntz_rel}) imply that $A_S$ is the Cuntz-Toeplitz algebra. By the structure of Cuntz-Toeplitz algebra indicated above and
Proposition \ref{compact_A_S} we get the following result:
\begin{equation*}
A_S/\mathcal{K}(H)\cong\mathcal{O}_s\,.
\end{equation*}
This implies that we have the following short exact sequence:
\begin{equation}\label{ses-As}
0\rightarrow\mathcal{K}(H)\rightarrow A_S\overset{q_S}{\rightarrow}\mathcal{O}_s\rightarrow 0
\end{equation}
where $q_S:A_S\to \mathcal{O}_s$ is the corresponding quotient map.

\subsection{The Toeplitz Map} 
We will define and study a very useful Toeplitz map 
\begin{equation*}
T_S:\mathcal{O}_s\to A_S.
\end{equation*}
To do that we  need the following concrete representation of the Cuntz algebra in the Hilbert space $\ell^2(\Z)$. It is defined by specifying $s$ orthogonal isometries given on standard basis elements $\{e_l\}_{l\in\Z}$ by the following formulas:
\begin{equation*}
u_je_l=e_{sl+j}, \quad\text{for } j=0,\ldots,s-1.
\end{equation*}
Formulas for the adjoints are:
\begin{equation*}
u_j^*e_l= \left\{
\begin{aligned}
&e_{(l-j)/s} &&\textrm{ if }l\equiv j\textrm{ mod }s\\
&0&&\textrm{ otherwise}.
\end{aligned}\right.
\end{equation*}
It is easy to verify that the relations in equation \eqref{Cuntzrelations} are satisfied.

Define the map $\phi:\mathcal{V}\to\Z$  by
\begin{equation*}
\phi(n,x):=s^n+x.
\end{equation*}
It is easily seen that $\phi$ is one-to-one. This map is used to define an isometry 
\begin{equation*}
\iota:\ell^2(\mathcal{V})\to \ell^2(\Z)
\end{equation*}
on basis elements by
\begin{equation*}
\iota(E_{(n,x)})=e_{\phi(n,x)}.
\end{equation*}
Its adjoint $\iota^*:\ell^2(\Z)\to \ell^2(\mathcal{V})$ is given by
\begin{equation*}
\iota^*e_l= \left\{
\begin{aligned}
&E_{\phi^{-1}(l)} &&\textrm{ if }l\in\textrm{ Ran }\phi\\
&0&&\textrm{ otherwise}.
\end{aligned}\right.
\end{equation*}
Our Toeplitz map is initially defined from $B( \ell^2(\Z))$ to $B(H)$ by the formula
\begin{equation*}
T_S(a):=\iota^* a \iota.
\end{equation*}
Clearly, $a\mapsto T_S(a)$ is linear and continuous, and we have $T_S(I)=I$ and $T_S(a)^*=T_S(a^*)$. The key property relevant for our discussion of the Cuntz and Cuntz-Toeplitz algebras is the following observation.
\begin{prop}\label{TSproduct}
\begin{equation*}
T_S(u_{(n,x)}u_{(m,y)}^*)=S_{(n,x)}S_{(m,y)}^*+C_{(n,x),(m,y)},
\end{equation*}
where $C_{(n,x),(m,y)}$ is non-zero only if both $x$ and $y$ are in the range of $\phi$. When $x$ and $y$ are in the range of $\phi$ so that $x=s^j+x'$, $y=s^l+y'$ where $j,l \ge 1, 0\le x' < s^j, 0\le y'<s^l$, then $C_{(n,x),(m,y)}$ is a one-dimensional, and hence compact, operator on $H$ given by
\begin{equation*}
C_{(n,x),(m,y)}E_{(k,z)}= \left\{
\begin{aligned}
&E_{(j,x')} &&\textrm{ if }k=l\textrm{ and }z=y'\\
&0&&\textrm{ otherwise}.
\end{aligned}\right.
\end{equation*}
\end{prop}
\begin{proof} We outline the key steps of this otherwise straightforward calculation.
We have the following easily established formulas:
\begin{equation*}
S_{(n,x)}E_{(k,z)}= E_{(k+n,s^nz+x)}
\end{equation*}
and
\begin{equation*}
S_{(n,x)}^*E_{(k,z)}= \left\{
\begin{aligned}
&E_{(k-n,(z-x)/s^n)} &&\textrm{ if }k\geq n\textrm{ and }s^n|(z-x)\\
&0&&\textrm{ otherwise}.
\end{aligned}\right.
\end{equation*}
Similarly, we have
\begin{equation*}
u_{(n,x)}e_l=e_{s^nl+x}
\end{equation*}
and
\begin{equation*}
u_{(n,x)}^*e_l= \left\{
\begin{aligned}
&e_{(l-x)/s^n} &&\textrm{ if }l\equiv x\textrm{ mod }s^n\\
&0&&\textrm{ otherwise}.
\end{aligned}\right.
\end{equation*}
The last formula implies that in the expression
\begin{equation*}
T_S(u_{(n,x)}u_{(m,y)}^*)E_{(k,z)}=\iota^* (u_{(n,x)}u_{(m,y)}^*) \iota E_{(k,z)}= \iota^* u_{(n,x)}u_{(m,y)}^* e_{\phi(k,z)},
\end{equation*}
the term $u_{(m,y)}^* e_{\phi(k,z)}$ is non-zero if and only if
\begin{equation*}
s^k+z\equiv y\textrm{ mod }s^m.
\end{equation*}
If $k<m$ the last condition means equality $y=s^k+z$, since $z<s^k$. In particular, if $k<m$, for that term to be non-zero $y$ must be in the range of $\phi$ and then
$$\iota^* u_{(n,x)}u_{(m,y)}^* e_{\phi(k,z)}=\iota^* u_{(n,x)}e_0=\iota^* e_x.$$
From the formula for $\iota^*$ the expression is non-zero if and only if $x$ is in the range of $\phi$. Consequently, this case leads to the one-dimensional operator $C_{(n,x),(m,y)}$. In all other cases we have
$T_S(u_{(n,x)}u_{(m,y)}^*)=S_{(n,x)}S_{(m,y)}^*$.\end{proof}
This result has several important consequences. We already noticed that the *-subalgebra Pol$(\mathcal{O}_s)$ of finite linear combinations of products $u_{(n,x)}u_{(m,y)}^*$ is dense in $\mathcal{O}_s$. Similarly, applying the calculations of Lemma V.4.1 of \cite{KD} to any product of $S_{(n,x)}$'s  and  $S_{(m,y)}^*$'s we see that it can be simplified so that no $S_{(n,x)}$'s are to the right of any $S_{(m,y)}^*$'s. This implies that the *-subalgebra Pol$(A_S)$ of finite linear combinations of products $S_{(n,x)}S_{(m,y)}^*$ is dense in $A_S$. Additionally, since $\mathcal{K}(H) \subseteq A_S$, we have that indeed the Toeplitz map $T_S$ maps $\mathcal{O}_s$ to $A_S$ by continuity.

If $q_S:A_S\to \mathcal{O}_s$ is the quotient map then we have that $q_S$ is zero on compact operators and
\begin{equation*}
q_S(S_{(n,x)}S_{(m,y)}^*)=u_{(n,x)}u_{(m,y)}^*,
\end{equation*}
for all $n$, $m$, $x$, $y$. It follows by continuity again that
\begin{equation*}
q_S(T_S(a))=a
\end{equation*}
for every $a\in \mathcal{O}_s$. Finally, the key importance of the Toeplitz map is in the following observation: for every $a,b\in \mathcal{O}_s$ we have
\begin{equation*}
T_S(ab)-T_S(a)T_s(b)\in \mathcal{K}(H).
\end{equation*}
Indeed, by continuity, it is enough to verify the above for $a,b$ of the form $S_{(n,x)}S_{(m,y)}^*$.
Invoking Lemma V.4.1 of \cite{KD} again the statement follows from Proposition \ref{TSproduct}.

\subsection{The Coefficient Algebra of $A_S$} 

Recall that the *-subalgebra Pol$(A_S)$ of finite linear combinations of products $S_{(n,x)}S_{(m,y)}^*$ is dense in $A_S$. This observation, and the gauge action on $S_{(n,x)}$'s  and  $S_{(m,y)}^*$'s suggest to consider the following subalgebra of $A_S$:
\begin{equation*}
B_S:= C^*(S_{(n,x)}S_{(n,y)}^*: n\in \Z_{\ge0}, 0\le x,y<s^n).
\end{equation*}

\begin{prop}\label{BSProp}
The subalgebra $B_S$ is the minimal coefficient algebra extension of $C(\Z_s)$ and the algebra $A_{S,\textrm{inv}}$ is equal to $B_S$. 
\end{prop}
\begin{proof} Clearly $B_S\subseteq A_{S,\textrm{inv}}$. To prove the other inclusion we consider $a\in A_{S,\textrm{inv}}$, let $\epsilon>0$ and choose $a_\epsilon\in \textrm{Pol}(A_S)$ such that:
\begin{equation*}
||a-a_\epsilon||<\epsilon.
\end{equation*}
Applying the expectation from equation \eqref{expect_formula} we obtain:
\begin{equation*}
||a-E(a_\epsilon)||=||E(a-a_\epsilon)||< ||a-a_\epsilon||<\epsilon.
\end{equation*}
This completes the proof since $E(a_\epsilon)\in B_S$.
\end{proof}

\subsection{Gauge Invariant Compact Operators} 

Next we  briefly discuss the structure of the algebra $\mathcal{K}(H)_{\mathrm{inv}}$ of gauge invariant compact operators, which is a subalgebra of $B_S$.

For $n=0,1,\ldots$, let $H_n$ be the finite dimensional subspace of $H$ spanned by basis vectors $E_{(n,x)}$ with $0\leq x< s^n$, so that the dimension of $H_n$ is $s^n$ and
\begin{equation*}
H=\bigoplus_{n=0}^\infty H_n.
\end{equation*}
Recall that from  definition \eqref{Udefref}, the operator $\mathcal{U}_\theta$ is a scalar multiplication by $e^{2\pi in\theta}$ on each $H_n$ . Consequently, if a bounded operator $a$ on $H$ is $\rho_\theta$ invariant and $m\ne n$ we have:
\begin{equation*}
(E_{(n,x)},aE_{(m,x)})=(E_{(n,x)},\rho_\theta(a)E_{(m,x)})=(\mathcal{U}_\theta^{-1}E_{(n,x)},a\, \mathcal{U}_\theta^{-1}E_{(m,x)})=0.
\end{equation*}
It follows that gauge invariant operators are block diagonal:
\begin{equation*}
a=\bigoplus_{n=0}^\infty a_n,
\end{equation*}
where $a_n:H_n\to H_n$ are finite rank operators. Such an operator is compact if and only if $||a_n||\to 0$ as $n\to\infty$. Consequently, the algebra 
$\mathcal{K}(H)_{\mathrm{inv}}$  of invariant compact operators is an AF algebra that can be written as a direct limit:
\begin{equation*}
\mathcal{K}(H)_{\mathrm{inv}}=\lim_{\underset{n\in \Z_{\geq 0}}{\longrightarrow}}\mathcal{A}_n\,,
\end{equation*}
where 
\begin{equation*}
\mathcal{A}_n:=\bigoplus_{j=0}^n M_{s^n}(\C),
\end{equation*}
and the connecting maps $\mathcal{A}_n\to \mathcal{A}_{n+1}$ are given by:
\begin{equation*}
(a_0, a_1,\ldots,a_n)\mapsto (a_0, a_1,\ldots,a_n, 0)
\end{equation*}

\subsection{$A_S$ as a Crossed Product} 
Consider the following subalgebra of $\mathcal{O}_s$:
\begin{equation*}
\mathcal{O}_{s,\mathrm{inv}}:= C^*(u_{(n,x)}u_{(n,y)}^*: n\in \Z_{\ge0}, 0\le x,y<s^n).
\end{equation*}
It is known, see \cite{Cu3}, that $\mathcal{O}_{s,\mathrm{inv}}$ is the invariant subalgebra of $\mathcal{O}_s$ with respect to the gauge action on $\mathcal{O}_s$ and that 
$$\mathcal{O}_{s,\mathrm{inv}}\cong UHF(s^\infty).$$

Using the above Toeplitz map $T_S:\mathcal{O}_s\to A_S$ we can give a more detailed description of the gauge invariant algebra $B_S$.
\begin{prop}\label{BSThm}
Every $b\in B_S$ can be uniquely written as
\begin{equation}\label{ToeplitzS}
b=T_S(a)+c,
\end{equation}
where $a\in  \mathcal{O}_{s,\mathrm{inv}}$ and $c\in  \mathcal{K}(H)_{\mathrm{inv}}$
Moreover, we have the following short exact sequence of C$^*$-algebras:
\begin{equation*}
0\rightarrow \mathcal{K}(H)_{\mathrm{inv}} \rightarrow B_S\rightarrow \mathcal{O}_{s,\mathrm{inv}} \rightarrow 0.
\end{equation*}
\end{prop}
\begin{proof}
We  begin by constructing the desired $a$.  Put $a:=q_S(b)\in\ocal_S$.  Then we have  
$$q_S(b-T_S(a)) = q_S(b)-q_S(T_S(a)) = q_S(b)-a = q_S(b) -q_S(b)=0.$$  
So we get $c:=b-T_S(a) \in \kcal(H)$.  

To obtain uniqueness, we show that if $T_S(a)\in \kcal(H)$ then $a=0$.  Supposing that $a\ne 0$, there are $a_1$ and $a_2$ in $\ocal_S$ so that $a_1aa_2 = I$.  Applying $T_S$ to the left side we get 
$$I=T_S(I)=T_S(a_1aa_2)=T_S(a_1)T_S(a)T_S(a_2) + \tilde{c}$$ 
where $\tilde{c}$ is compact.  But $T_S(a)$ is compact by assumption, thus the above expression is compact.  This implies that $I$ is compact, an obvious contradiction.  

Finally, notice that the short exact sequence in equation \eqref{ses-As} induces the short exact sequence of invariant algebras.
\end{proof}

Since $B_S$ is equal to $A_{S,\textrm{inv}}$ the map $\alpha_S$ is a monomorphism of $B_S$ with a hereditary range.  Similarly to the previous sections we obtain the following main structural result as a direct consequence of Proposition \ref{BSProp} and Proposition \ref{CPProp}.

\begin{theo}\label{ASTheo}
We have the following isomorphisms of C$^*$-algebras:
\begin{equation*}
A_S \cong B_S\rtimes_{\alpha_S}\N \,.
\end{equation*}
\end{theo}

\section{The Serre Shift}
\subsection{The Algebra $A_W$}
In this final section, the goal is to study the algebra 
\begin{equation*}
A_W=C^*(W,M_f:f\in C(\Z_s)).
\end{equation*}   Recall that the Serre Shift is given by:
\begin{equation*}
WE_{(n,x)} = \frac{1}{\sqrt{s}}\sum_{j=0}^{s-1}E_{(n+1,\,x+js^n)}\,.
\end{equation*}

As with the other algebras, we describe a relation between $W$, $W^*$ and $M_f$.  We have the following lemma:

\begin{lem}\label{partial_W_rel}
The operators $W$, $W^*$, and $M_f$ satisfy the following relations on the basis elements of $H$:
\begin{equation*}
\begin{aligned}
W^*M_fWE_{(n,x)} &= \frac{1}{s}\sum_{j=0}^{s-1}f(x+js^n)E_{(n,x)}\\
WM_fW^*E_{(n,x)} &=\frac{1}{s}f(x\textrm{ mod }s^{n-1})WW^*E_{(n,x)}
\end{aligned}
\end{equation*}
where
\begin{equation*}
WW^*E_{(n,x)} = \left\{
\begin{aligned}
&\frac{1}{s}\sum_{j=0}^{s-1}E_{(n,x\textrm{ mod }s^{n-1}+js^{n-1})} &&\textrm{ if }n\ge 1\\
&0 &&\textrm{ else.}
\end{aligned}\right. 
\end{equation*}
\end{lem}
\begin{proof}
The formulas follow by direct calculations similar to those in Lemmas \ref{A_U_relations}, \ref{A_V_relations}, and \ref{A_S_relations}.
\end{proof}

In light of this lemma it does not appear, a priori, that there is an automorphism or even an endomorphism of $C(\Z_s)$ associated to $A_W$ to construct a crossed product.  The right hand side of the above formulas suggest however the following:  if $F:\mathcal{V}\to \C$ is a bounded function, define 
\begin{equation*}
\begin{aligned}
\mathfrak{a}_WF(n,x) &= \left\{
\begin{aligned}
&F(n-1,\,x\textrm{ mod }s^{n-1}) &&\textrm{ if }n\ge1\\
&0 &&\textrm{ else}
\end{aligned}\right. \\
\mathfrak{b}_WF(n,x) &= \frac{1}{s}\sum_{j=0}^{s-1}F(n+1,\,x+js^n)\,.
\end{aligned}
\end{equation*} 
Moreover, for a bounded function $F$, define the operator $M_F$ acting on $H$ by 
\begin{equation*} 
M_FE_{(n,x)}=F(n,x)E_{(n,x)}.
\end{equation*}
With those considerations, similarly to Lemma \ref{partial_W_rel} we get
\begin{equation}\label{W_rel}
W^*M_FW=M_{\mathfrak{b}_WF}\quad\textrm{and}\quad WM_F=M_{\mathfrak{a}_WF}W\,.
\end{equation}
The next goal is to identify for which $F$ the corresponding diagonal operator $M_F$ is in $A_W$.

\subsection{The Diagonal Subalgebra}
Consider the space of complex-valued functions $\mathcal{F}$ defined on $\mathcal{V}$  as:
\begin{equation*}
\{F:\mathcal{V}\to\C :\textrm{there is an }f_F\in C(\Z_s), \textrm{ so that }\lim_n\underset{0\le x<s^n}{\textrm{sup }}|F(n,x\textrm{ mod }s^n)-f_F(x)|=0\}\,.
\end{equation*}
It is not difficult to see that $\mathcal{F}$ is a commutative C$^*$-algebra with the usual supremum norm. 

Define the following subset of $\mathcal{F}$
\begin{equation*}
\mathcal{F}_\infty = \{F\in\mathcal{F} : \lim_n\underset{0\le x<s^n}{\textrm{sup }}|F(n,x)|=0\}\,.
\end{equation*}
Notice that $\mathcal{F}_\infty$ is a $*$-subalgebra and an ideal of $\mathcal{F}$.  Also, if the set 
\begin{equation*}
\mathcal{V}=\{(n,x): n=0,1,2,\ldots,0\le x<s^n\}
\end{equation*} 
is equipped with the discrete topology, then we notice that $\mathcal{F}_\infty$ can be seen as the space of continuous functions vanishing at infinity on $\mathcal{V}$:
\begin{equation*}
\mathcal{F}_\infty = C_0(\mathcal{V})\,.
\end{equation*}

Clearly, any $F\in\mathcal{F}$ can be uniquely decomposed as follows:
\begin{equation*}
F(n,x) = f_F(x) + \varphi(n,x)
\end{equation*}
where $f_F\in C(\Z_s)$ and $\varphi\in\mathcal{F}_\infty$. Moreover we have the following isomorphism:
\begin{equation*}
\mathcal{F}/\mathcal{F}_\infty\cong C(\Z_s)\,.
\end{equation*}

The main reason we are interested in the algebra $\mathcal{F}$ is explained in the following result.

\begin{prop}
For every $F\in\mathcal{F}$ we have $M_F\in A_W$.
Moreover, $\mathcal{F}$ is the smallest commutative C$^*$-algebra of bounded functions on $\mathcal{V}$ such that the following conditions are satisfied: 
\begin{description}
\item[(1)] $\mathcal{F}$ contains all of the functions $(n,x)\mapsto f(x)$ with $f\in C(\Z_s)$
\item[(2)] $\mathcal{F}$ is closed under the maps $\mathfrak{a}_W$ and $\mathfrak{b}_W$.
\end{description}
\end{prop}
\begin{proof}
Notice that by the definition of $\mathfrak{a}_W$ for $F\in \fcal$ we have 
\begin{equation*}
|\mathfrak{a}_WF(n,x)-f_F(x)|\leq|F(n-1,\,x\textrm{ mod }s^{n-1}) - f_F(x\textrm{ mod }s^{n-1})|+|f_F(x) - f_F(x\textrm{ mod }s^{n-1})|.
\end{equation*}
From the definition of $f_F$, the first term above goes to zero as $n\to\infty$. By continuity of $f_F$, the second term also goes to zero as $n\to\infty$.
Similar estimates also work for $\mathfrak{b}_W F$. We obtain:
\begin{equation}\label{aWbWlimit}
f_{\mathfrak{a}_WF} = f_{\mathfrak{b}_WF} = f_F.
\end{equation}
Thus we have that $\fcal$ is closed under $\mathfrak{a}_W$ and $\mathfrak{b}_W$.

Let $(g_n)_{n=0}^\infty$ with $g_n\in\mathcal{F}_\infty$ be defined by \begin{equation*}
g_n(m,x) = \begin{cases} 1 & n=m \\ 0 & \text{else}.\end{cases} 
\end{equation*} 
Since the set $\left\{ g_n f : f\in C(\Z_s), n=0,1,\ldots\right\}$ separates  points of $\mathcal{V}$ and has the non-vanishing property, since $g_n(n,x)=1\ne 0$, the Stone-Weierstrass Theorem yields that 
$$C^*(g_nf : f\in C(\Z_s), n=0,1,\ldots) =\mathcal{F}_\infty.$$
Thus it suffices to show that $M_{g_n}\in A_W$ for all $n\ge 0$.  We proceed inductively.  

If we put 
$$\chi_1(x)=e^{2\pi i x/s} \in C(\Z_s),$$ 
then we have
\begin{equation*}
W^* M_{\chi_1}W  = M_{\mathfrak{b}_W\chi_1} = M_{\chi_1}(1-M_{g_0}).
\end{equation*}
Thus $M_{g_0}\in A_W$ since $ M_{\chi_1}$ is invertible.

Assume $M_{g_0}, M_{g_1},\ldots,M_{g_{n-1}} \in A_W$.  It follows that $M_g\in A_W$ for all $g\in c_0(\Z_{\ge 0})$ with $g(m)=0$ for $m\ge n$.  This is because such $g$'s form a $n$-dimensional space and functions $\{g_0,\ldots,g_{n-1}\}$ are a basis for this space.

For $n=0,1,\ldots$, define
$$\chi_n(x)=e^{2\pi ix/s^n}$$ 
and note that we have 
$$W^*M_{\chi_n}W = M_{\mathfrak{b}_W\chi_n}.$$ 
Furthermore, we get
\begin{equation*}
\mathfrak{b}_W\chi_{n+1}(m,x)=\dfrac{1}{s}\sum_{j=0}^{s-1} e^{2\pi i\left(\frac{x+js^m}{s^{n+1}}\right)}
=\chi_{n+1}(x) \dfrac{1}{s} \sum_{j=0}^\infty e^{2\pi i j\left(\frac{s^m}{s^{n+1}}\right)}.
\end{equation*}
Consider the function on the right-hand side of the above equation:
\begin{equation*}
h_n(m):=\dfrac{1}{s} \sum_{j=0}^\infty e^{2\pi i j\left(\frac{s^m}{s^{n+1}}\right)}.
\end{equation*}
It follows from the above calculation and invertibility of $M_{\chi_n}$ that $M_{h_n}\in A_W$.
Notice that $h_n$ has the following property:
\begin{equation*}
h_n(m)=\begin{cases}
1 & m>n \\ 0 & m=n.
\end{cases} 
\end{equation*} 
Then the function $(1-h_n-g_n)(x)=0$ for $m\ge n$, and thus by induction hypothesis $M_{1-h_n-g_n} \in A_W$. Hence we have $M_{g_n}\in A_W$, completing the induction.

To show that $\mathcal{F}$ is the smallest C$^*$-algebra of bounded functions on $\mathcal{V}$ satisfying the two conditions notice any such algebra would necessarily contain the functions $\mathfrak{b}_W(\chi_n)$. But the above proof implies then that such an algebra must contain all  $g_nf$ with $f\in C(\Z_s)$ for $n=0,1,\ldots$ and consequently must contain $\mathcal{F}$.
\end{proof}

As a consequence, we get the following equality of C$^*$-algebras:

\begin{prop}
The $C^*$-algebra $\textrm{C}^*(W,M_F)$ generated by $W$ and $M_F$ with $F\in\mathcal{F}$ is equal to $A_W$.
\end{prop}
\begin{proof}
Since $C(\Z_s)\subseteq\mathcal{F}$, it follows that $A_W\subseteq \textrm{C}^*(W,M_F)$.
From the previous proposition, for every $F\in\mathcal{F}$ we have $M_F\in A_W$, which implies that $\textrm{C}^*(W,M_F)\subseteq A_W$. 
\end{proof}

\subsection{The Coefficient Algebra of $A_W$}
Next consider the mutually orthogonal projections $\{\mathcal{P}_n\}$ given by 
\begin{equation*} 
\mathcal{P}_0=I-WW^*\textrm{ and } \mathcal{P}_n = W^n\mathcal{P}_0(W^*)^{n}.
\end{equation*} 
Below we show that the C$^*$-algebra $B_W$ generated by $M_F$'s and $\{\mathcal{P}_n\}$:
\begin{equation*}
B_W:=C^*(M_F,\mathcal{P}_n: F\in\mathcal{F}, n\in\Z_{\geq 0})
\end{equation*}
is the minimal coefficient algebra extension of $C(\Z_s)$ and so the algebra $A_{W,\textrm{inv}}$ is equal to $B_W$, see Proposition \ref{BWProp} below. While it is clear that $B_W\subseteq A_{W,\textrm{inv}}$, it is not evident that $B_W$ is a coefficient algebra. This is because the map $\beta_W$ is not a homomorphism and thus verifying that $B_W$ is closed under $\beta_W$ requires more structural work than just inspecting the generators of $B_W$.

The result that $A_W=C^*(W, M_F:F\in\mathcal{F})$ allows for extra flexibility. We use it in the proof of the following observation.  
\begin{prop}
The set $\mathcal{K}(H)$ of compact operators on $H$ is an ideal of $A_W$.
\end{prop}

\begin{proof}
To show that $A_W$ contains the ideal of compact operators $\mathcal{K}$ we proceed to show the system of units that generate $\mathcal{K}(H)$ are in $A_W$, similarly to an argument in the proof of Proposition \ref{compact_A_S} for the Bernoulli shift. 

Define $\chi_{m,l}:\mathcal{V}\to\C$ by 
$$\chi_{m,l}(n,x) = \begin{cases}
1 & n=m \text{ and } |x-l|_s \le s^{-m} \\ 0 & \text{else}\end{cases}.$$  
Since $\chi_{m,l}(n,x)=0$ for $n>m$, we get that 
$$\lim\limits_{n\to\infty}\sup\limits_{0\le x < s^n} |\chi_{m,l}(n,x)-0| = 0,$$ 
hence $\chi_{m,l}\in \mathcal{F}$.  
The units $P_{(n,x),(m,y)}$ that generate $\mathcal{K}(H)$ are defined by 
$$P_{(n,x),(m,y)} = s^{\frac{n+m}{2}} M_{\chi_{n,x}} W^n (W^*)^m M_{\chi_{m,y}},$$  
where $n\in\Z_{\geq 0}$, $0\leq x<s^n$, $m\in\Z_{\geq 0}$ and $0\leq y<s^m$.
To verify this claim, we calculate $P_{(n,x),(m,y)}E_{(k,z)}$ in steps:
\begin{equation*}
\begin{aligned}
s^{m/2} (W^*)^m M_{\chi_{m,y}}E_{(k,z)} & = \begin{cases}
s^{m/2} (W^*)^m E_{(m,z)} & k=m, z\equiv y\bmod s^m \\ 0 & \text{else}
\end{cases} \\ &=\begin{cases}
E_{(0,0)} & k=m, z\equiv y\bmod s^m \\ 0 & \text{else}.
\end{cases}
\end{aligned}
\end{equation*}
But $0 \le z, y < s^m$, so the only way for these to be congruent is if they are equal.  Thus
\begin{equation*}
s^{m/2} (W^*)^m M_{\chi_{m,y}}E_{(k,z)} = \begin{cases}
E_{(0,0)} & k=m, y=z \\ 0 & \text{else}.
\end{cases}
\end{equation*} 
The remaining portion of $P_{(n,x),(m,y)}E_{(k,z)}$ is:
\begin{equation*}
\begin{aligned}
s^{n/2} M_{\chi_{n,x}}W^n E_{(0,0)} &= s^{n/2} M_{\chi_{n,x}} \dfrac{1}{s^{n/2}} \sum_{j_0=0}^{s-1}\sum_{j_1=0}^{s-1}\cdots\sum_{j_{n-1}=0}^{s-1} E_{(n,j_0+j_1 s + \cdots + j_{n-1}s^{n-1})} \\
&= M_{\chi_{n,x}} \sum_{j=0}^{s^n-1} E_{(n,j)} = E_{(n,x)}.
\end{aligned}
\end{equation*}
Thus we obtain:
\begin{equation*}
P_{(n,x),(m,y)} E_{(k,z)} = \begin{cases}
E_{(n,x)} & (m,y)=(k,z) \\ 0 & \text{else}.
\end{cases}
\end{equation*}
These are precisely the matrix units for the canonical base of $H$, and so $\mathcal{K}(H)\subseteq A_W$.
\end{proof}

Notice that the formula
$$P_{(n,x),(n,y)} = s^{n} M_{\chi_{n,x}} W^n (W^*)^n M_{\chi_{n,y}}$$  
from above implies that $P_{(n,x),(n,y)}$ are in $B_W$ for all $n=0,1,\ldots$ and $0\leq x,y<s^n$. This is because $\chi_{n,x}\in \mathcal{F}$ and 
\begin{equation*}
W^n (W^*)^n = I-\sum_{j=0}^{n-1} \mathcal{P}_j.
\end{equation*}
Consequently, $\mathcal{K}(H)_{\mathrm{inv}}$ is a subalgebra, in fact an ideal in $B_W$. 

Notice also that if $F\in \mathcal{F}_\infty=C_0(\mathcal{V})$ then $M_F\in \mathcal{K}(H)_{\mathrm{inv}}$ since $M_F$ is a diagonal operator with coefficients vanishing at infinity.
In particular, for any $ F\in \mathcal{F}_\infty$, we have that $M_{F-\mathfrak{a}_W(F)}$ is compact, an observation used several times below.

Recall that we can view $c(\Z_{\ge0}) \otimes C(\Z_s) $ as the space of uniformly convergent sequences of continuous functions on $\Z_s$:
\begin{equation*}
c(\Z_{\ge0}) \otimes C(\Z_s)\cong c(\Z_{\ge0}, C(\Z_s))\cong \{G=(g_n)_{n\in \Z_{\ge0}}: g_n\in C(\Z_s),\  \lim_{n\to\infty}g_n\textrm{ exists}\}\,,
\end{equation*}
with the following notation for the uniform limit:
$g_\infty:=\lim_{n\to\infty}g_n$.

Consider the following Toeplitz-like map
\begin{equation*}
c(\Z_{\ge0}, C(\Z_s))\ni G\mapsto T_W(G):= \sum_{n=0}^{\infty}\mathcal{P}_n(M_{g_n}-M_{g_\infty})\mathcal{P}_n+M_{g_\infty} \in B_W.
\end{equation*}
Mutual orthogonality of the $\pcal_n$'s ensures that the series is convergent because 
\begin{equation*}
\left\|\sum_{n=0}^\infty \pcal_n (M_{g_n}-M_{g_\infty})\pcal_n\right\| = \sup_n \|M_{g_n}-M_{g_\infty}\| < \infty
\end{equation*}
since $g_n \to g_\infty$. This Toeplitz map is the key to describing the structure of $B_W$. 

We need the following simple but essential information.
\begin{lem} \label{WCommLemma}
For every $F\in\mathcal{F}$ the commutator $[M_F,W]$ is compact. For every $F\in\mathcal{F}$ and $n\in\Z_{\geq0}$ the commutator $[M_F,\pcal_n]$ is compact.
\end{lem}
\begin{proof}
\begin{equation*}
[M_F,W] = M_F W-WM_F= M_FW-M_{\mathfrak{a}_W(F)}W=M_{F-\mathfrak{a}_W(F)}W,
\end{equation*}
and $M_{F-\mathfrak{a}_W(F)}$ is a compact diagonal operator.

The second statement follows from the first since the map $a\mapsto [M_F,a]$ is a derivation and so it satisfies the Leibniz Rule.
\end{proof}

We use Lemma \ref{WCommLemma} to verify the following crucial property of the Toeplitz map $T_W$.

\begin{prop} \label{twprod}
For all $G,G'\in c(\Z_{\ge0}, C(\Z_s))$ we have
\begin{equation*}
T_W(G)T_W(G')-T_W(GG')\in  \mathcal{K}(H)_{\mathrm{inv}}.
\end{equation*}
\end{prop}
\begin{proof} By continuity, it is sufficient to check the case where $T_W(G)$ is a finite sum, that is, where the sequence $\{g_n\}$ is eventually constant.  In this case, we have
\begin{equation*}
T_W(G) =\sum_{n=0}^{N} \pcal_n (M_{g_n}-M_{g_\infty})\pcal_n + M_{g_\infty} \\
\end{equation*}
where $g_n=g_\infty$ for $n>N$. Using Lemma \ref{WCommLemma} this expression can be rewritten as follows:
\begin{equation}\label{twgalt}
T_W(G) = \sum_{n=0}^N \pcal_n M_{g_n}\pcal_n + \pcal_{\le N}^\perp M_{g_\infty}\pcal_{\le N}^\perp + c,
\end{equation}
where $c$ is compact and 
$$\pcal_{\le N}^\perp=I- \sum_{n=0}^N \pcal_n.$$

If we consider two eventually constant sequences $G,G'$, take $N$ large enough so that both $g_n$ and $g_n'$ are constant for $n>N$. Then it follows from  formula \eqref{twgalt} and Lemma \ref{WCommLemma} that we have $T_W(GG') - T_W(G)T_W(G') \in \kcal(H)_{\mathrm{inv}}$.
\end{proof}

The following result describes the structure of the algebra $B_W$.

\begin{prop}\label{BWThm}
Every $a\in B_W$ can be uniquely written as
\begin{equation}\label{ToeplitzW}
a=T_W(G)+c,
\end{equation}
for some $c\in \mathcal{K}(H)_{\mathrm{inv}}$ and $G\in  c(\Z_{\ge0}) \otimes C(\Z_s)$.
\end{prop}
\begin{proof}   
Define 
\begin{equation*}
B_W':= \left\{T_W(G)+c : G\in c(\Z_{\ge 0},C(Z_s)),\, c\in \kcal(H)_{\mathrm{inv}}\right\} \subseteq A_{W,\mathrm{inv}}.
\end{equation*}
First we notice that $B_W'$ is a $*$-subalgebra of $A_W$.  It is clearly a subspace of $A_W$ and is closed under adjoints, thus it is sufficient to show that $B_W'$ is closed under products, but this immediately follows from Proposition \ref{twprod}.

Next we demonstrate that $B_W'=B_W$. We already noticed that for all $G\in c(\Z_{\ge0}, C(\Z_s))$ we have $T_W(G)\in B_W$ and also $\kcal(H)_{\mathrm{inv}}\subseteq B_W$.  So $B_W' \subseteq B_W$.

For the reverse inclusion we only need to check that the generators of $B_W$ are in $B_W'$.  For any $F\in \mathcal{F}$ we decompose 
\begin{equation*}
M_F = M_{f_F} + M_{F-f_F}
\end{equation*}
and notice that $F-f_F\in \mathcal{F}_\infty$ and therefore $M_{F-f_F}\in\kcal(H)_{\mathrm{inv}}$.

Let $G_F\in  c(\Z_{\ge 0},C(\Z_s))$ be a constant sequence where each term is equal to $f_F$. 
Then we have
\begin{equation*}
T_W(G_F) = M_{f_F},
\end{equation*}
which implies that $M_F\in B_W'$.

Now fix $n\in \Z_{\ge 0}$ and put $G_n = (0,\ldots,0,1,0,\ldots)$ where the lone $1$ is in the $n^{th}$ place.  Then 
\begin{equation*}
T_W(G_n) = \pcal_n(1-0)\pcal_n = \pcal_n,
\end{equation*}
and so $\pcal_n\in B_W'$, establishing the required inclusion. In particular, every $a\in B_W$ can be written as a sum $a=T_W(G)+c$,
for some $c\in \mathcal{K}(H)_{\mathrm{inv}}$ and $G\in  c(\Z_{\ge0}, C(\Z_s))$.

Uniqueness comes from the fact that $T_W(G)\in \kcal(H)$ implies that $G=0$.  Indeed, assume
$$a:=T_W(G) = \sum_{n=0}^{\infty} \pcal_n (M_{g_n}-M_{g_\infty})\pcal_n + M_{g_\infty} \in \kcal(H).$$  
Then
\begin{equation*}
\begin{aligned}
\pcal_m  a \pcal_m = \pcal_m (M_{g_m}-M_{g_\infty})\pcal_m + \pcal_m M_{g_\infty} \pcal_m
\end{aligned}
\end{equation*}
is also compact, and so $\pcal_m M_{g_m} \pcal_m $ is compact.  By Lemma \ref{WCommLemma} this is equivalent to $\pcal_m M_{g_m}$ being compact.

To proceed further we compute diagonal inner products:
\begin{equation*}
\pcal_0E_{(n,x)} = \begin{cases} E_{(n,x)}-\dfrac{1}{s}\sum\limits_{j=0}^{s-1} E_{(n,x\bmod s^{n-1}+js^{n-1})} & n\ge 1 \\ 0 & \text{else}\end{cases}
\end{equation*}
and so
\begin{equation*}
\langle E_{(n,x)},\pcal_0 E_{(n,x)}\rangle = 1-\frac{1}{s}.
\end{equation*}
Consequently $\pcal_0$ cannot be compact.  Moreover, by the definition of $\pcal_m$, we have that $\pcal_0 = (W^*)^m \pcal_m W^m$, thus if $\pcal_m$ were compact this would force $\pcal_0$ to be compact.  Hence $\pcal_m$ is not compact.


If $f\in C(\Z_s)$ then $M_f E_{(n,x)} = f(x)E_{(n,x)}$.  This does not depend on $n$ and we have: 
\begin{equation*}
\pcal_0M_f E_{(n,x)} = f(x)\pcal_0 E_{(n,x)}
\end{equation*}
As above, we have that 
\begin{equation*}
\langle E_{(n,x)} ,\pcal_0 M_f E_{(n,x)}\rangle = f(x)\left(1-\frac{1}{s}\right).
\end{equation*}
So $\pcal_0 M_f$ not compact unless $f\equiv 0$.  

Consider the following calculation:
\begin{equation*}
(W^*)^m \pcal_m M_f W^m = \pcal_0 (W^*)^m M_f W^m = \pcal_0 M_{\mathfrak{b}_W^m(f)} = \pcal_0 M_f + \pcal_0 M_{\mathfrak{b}_W^m(f) - f}.
\end{equation*}
Notice that by equation \eqref{aWbWlimit} the term  $\pcal_0 M_{\mathfrak{b}_W^m(f) - f}$ is compact.  If $\pcal_m M_f$ were compact, the above equation would imply that $\pcal_0 M_f$ is compact.  Hence again, $\pcal_m M_f$ is not compact unless $f\equiv 0$.
%

To summarize, if $T_W(G) = \sum_{n=0}^{\infty} \pcal_n (M_{g_n}-M_{g_\infty})\pcal_n + M_{g_\infty}$ is in $\kcal(H)$ then $\pcal_m M_{g_m}$  must be as well. But we have from above that this is never compact unless $g_m=0$ for all $m$.  So $T_W(G)$ is compact if and only if $G=0$. This implies the uniqueness in the decomposition of the elements of $B_W$.
\end{proof}

As a consequence, we obtain the following properties of the subalgebra $B_W$.

\begin{prop}\label{BWProp}
The algebra $B_W$ is the minimal coefficient algebra extension of $C(\Z_s)$ and so the algebra $A_{W,\textrm{inv}}$ is equal to $B_W$. Moreover, we have the following short exact sequence of C$^*$-algebras:
\begin{equation}\label{BWshortexact}
0\rightarrow \mathcal{K}(H)_{\mathrm{inv}} \rightarrow B_W \rightarrow c(\Z_{\ge0})\otimes C(\Z_s)\rightarrow 0.
\end{equation}
\end{prop}
\begin{proof} 
Clearly $B_W\subseteq A_{W,\textrm{inv}}$. To prove the other inclusion we need that $B_W$ is a coefficient algebra.  We use here the results from Proposition \ref{BWThm}. Thus, we need to show that for $T_W(G)+c\in B_W$, both $W(T_W(G)+c)W^*$ and $W^*(T_W(G)+c)W$ are in $B_W$. As $c$ is compact and the compact operators form an ideal, we will drop the $c$ from the calculations below. 

We have the following helpful formulas:
\begin{equation}
\begin{aligned} 
W\pcal_n = WW^n \pcal_0 (W^*)^n(W^*W) = W^{n+1}\pcal_0 (W^*)^{n+1}W = \pcal_{n+1}W\\
W^*\pcal_n = W^* W^n \pcal_0 (W^*)^n = W^{n-1}\pcal_0 (W^*)^{n-1}W^* = \pcal_{n-1}W^*,
\end{aligned}
\end{equation}
where again we use $\pcal_{-1}=0$. Additionally, equation \eqref{W_rel} implies that for $F\in\mathcal{F}$ we have
\begin{equation*}
\beta_W(M_F)=M_{\mathfrak{b}_WF}\quad\textrm{and}\quad \alpha_W(M_F)=M_{\mathfrak{a}_WF}(I-\pcal_0)\,.
\end{equation*}

Since eventually constant sequences with values of $C(\Z_s)$ are dense in $c(\Z_{\ge 0},C(\Z_s))$ it is enough to consider such sequences.  Applying those formulas for an eventually constant sequence $G\in  c(\Z_{\ge0}, C(\Z_s))$ we get:
\begin{equation*}
\begin{aligned}
&\alpha_W(T_W(G))=W(T_W(G))W^*  \\
&=\sum_{n=0}^N \pcal_{n+1}(M_{\mathfrak{a}_W(g_n)}-M_{\mathfrak{a}_W(g_\infty)})(I-\pcal_0)\pcal_{n+1} + M_{\mathfrak{a}_W(g_\infty)}(I-\pcal_0)\\
&=\sum_{n=0}^N \pcal_{n}(M_{\mathfrak{a}_W(g_{n-1})}-M_{\mathfrak{a}_W(g_\infty)})\pcal_{n} + M_{\mathfrak{a}_W(g_\infty)}
\end{aligned}
\end{equation*}
where $g_{-1}=0$. Since for $f\in C(\Z_s)$ we have that $M_f-M_{\mathfrak{a}_W(f)}$ is compact by equation \eqref{aWbWlimit}, we obtain
\begin{equation*}
\alpha_W(T_W(G))=T_W(G')+c
\end{equation*}
where $c$ is compact and $G'=(g_n')_{n\in \Z_{\ge0}}$ with $g_n'=g_{n-1}$.

A similar argument shows that $\beta_W(T_W(G))=W^*T_W(G)W\in B_W$.
Thus $B_W$ is a coefficient algebra, establishing the final part of the equality.

Because $\mathcal{K}(H)_{\mathrm{inv}}$ is an ideal in $B_W$ we have the short exact sequence:
\begin{equation*}
0\rightarrow \mathcal{K}(H)_{\mathrm{inv}} \rightarrow B_W \rightarrow B_W/\mathcal{K}(H)_{\mathrm{inv}} \rightarrow 0.
\end{equation*}
The short exact sequence in equation \eqref{BWshortexact} follows now from Proposition \ref{twprod} and Proposition \ref{BWThm} which establish that the map:
\begin{equation*}
c(\Z_{\ge0}, C(\Z_s))\ni G\mapsto T_W(G)+\mathcal{K}(H)_{\mathrm{inv}}\in B_W/\mathcal{K}(H)_{\mathrm{inv}}
\end{equation*}
is an isomorphism of C$^*$-algebras.
\end{proof}

Since $B_W$ is equal to $A_{W,\textrm{inv}}$ by  Proposition \ref{BWProp}, we have that the map $\alpha_W$ is a monomorphism of $B_W$ with a hereditary range.  Similarly to the previous sections we obtain the following main structural result about the algebra $A_W$ as a direct consequence of Proposition \ref{CPProp}.

\begin{theo}\label{AWTheo}
We have the following isomorphisms of C$^*$-algebras:
\begin{equation*}
A_W \cong B_W\rtimes_{{\alpha}_W}\N \,.
\end{equation*}
\end{theo}

\end{document}